\documentclass{article}
\usepackage[utf8]{inputenc}
\usepackage{todonotes}
\usepackage{hyperref}
\usepackage{bbm}
\usepackage{amsfonts}
\usepackage{amsmath}
\usepackage{amssymb}
\usepackage{mathtools}
\usepackage{mathabx}
\usepackage{amsthm}
\usepackage{cleveref}
\usepackage{xfrac}
\usepackage{float}
\usepackage{svg}
\usepackage[title]{appendix}
\usepackage[a4paper, margin=1in]{geometry}
\usepackage{xspace}
\usepackage{multirow}
\usepackage{comment}
\usepackage{authblk}
\usepackage{enumitem}
\usepackage{stackengine}
\usetikzlibrary {patterns,patterns.meta} 
\usetikzlibrary{fit}
\usetikzlibrary{arrows, arrows.meta}
\usetikzlibrary {graphs}
\tikzgraphsset{
empty nodes,
nodes={circle, draw, inner sep=0pt, minimum size=2pt},
counterclockwise, radius=8pt, phase=210,
}
\tikzset{
baseline=-2pt,
every label/.style={font=\tiny, inner sep=0pt},
dashed/.style={dash pattern=on 1pt off 1pt},
label distance=1pt,
}

\title{The Rank-Ramsey Problem and the Log-Rank Conjecture}
\author[1]{Gal Beniamini}
\author[1]{Nati Linial}
\author[2]{Adi Shraibman}
\affil[1]{The Hebrew University of Jerusalem}
\affil[2]{Tel Aviv-Yaffo Academic College}
\date{}

\hypersetup{colorlinks=true}	

\newtheorem{theorem}{Theorem}[section]
\newtheorem*{theorem*}{Theorem}
\newtheorem*{definition*}{Definition}
\newtheorem{corollary}{Corollary}[theorem]
\newtheorem{lemma}[theorem]{Lemma}
\newtheorem{proposition}[theorem]{Proposition}
\newtheorem{remark}[theorem]{Remark}
\newtheorem{claim}[theorem]{Claim}
\newtheorem{definition}[theorem]{Definition}

\newtheorem{openq}{Open Problem}[subsection]

\theoremstyle{plain}
\newtheorem{thm}{Theorem}

\newcommand{\NN}{\mathbb{N}}
\newcommand{\ZZ}{\mathbb{Z}}

\newcommand{\RR}{\mathbb{R}}

\DeclareMathOperator{\spec}{spec}

\DeclareMathOperator{\srg}{srg}

\DeclareMathOperator{\diam}{diam}

\DeclareMathOperator{\spar}{spar}

\newcommand{\widesim}[2][1.5]{
  \mathrel{\overset{#2}{\scalebox{#1}[1]{$\sim$}}}
}

\DeclareMathOperator{\NAE}{NAE}
\DeclareMathOperator{\AllEq}{AE}

\DeclareMathOperator*{\Motimes}{\text{\raisebox{0.25ex}{\scalebox{0.8}{$\bigotimes$}}}}

\DeclareMathOperator{\AND}{AND}

\DeclareMathOperator{\rank}{rank}

\DeclareMathOperator{\DCC}{D}

\newcommand{\eqdef}{\vcentcolon=}

\DeclareMathOperator{\mon}{mon}

\newcommand{\norm}[1]{\left\lVert#1\right\rVert}

\DeclareMathOperator{\msr}{msr}

\DeclareMathOperator{\GF}{GF}

\DeclareMathOperator{\Cay}{Cay}

\theoremstyle{remark}

\providecommand{\remarkname}{Remark}

\begin{document}

\maketitle

\vspace{-0.105cm}
\begin{abstract}
A graph is called Rank-Ramsey if (i) Its clique number is small, and (ii) The
adjacency matrix of its complement has small rank.
We initiate a systematic study of such graphs. Our main motivation is that 
their constructions, as well as proofs of their non-existence, are intimately related to
the famous log-rank conjecture from the field of communication complexity.
These investigations also open interesting new avenues in Ramsey theory.

We construct two families of Rank-Ramsey graphs exhibiting 
\textit{polynomial separation} between order and complement rank. Graphs in
the first family have bounded clique number (as low as $41$).
These are subgraphs of certain strong products,
whose building blocks are derived 
from triangle-free strongly-regular graphs.
Graphs in the second family are obtained by applying Boolean functions to Erd\H{o}s-R\'enyi graphs.
Their clique number is logarithmic, but their
complement rank is far smaller than in the first family, about $\mathcal{O}(n^{2/3})$.
A key component of this construction is our matrix-theoretic view of lifts.

We also consider lower bounds on the Rank-Ramsey numbers, and determine them in
the range where the complement rank is $5$ or less.
We consider connections between said numbers and 
other graph parameters, and 
find that the two best known explicit constructions of triangle-free Ramsey graphs 
turn out to be far from Rank-Ramsey.
\end{abstract}

\section{Introduction}
\label{section:introduction}

A graph $G$ is called \textit{Rank-Ramsey} if both its clique number and the rank of its complement are small.\footnote{As usual,
the rank of the graph $G$ is the real rank of its adjacency matrix $A_G$.}
Rank-Ramsey graphs are clearly Ramsey graphs, because
$\alpha(G) \le \rank(\overline{G})$
holds for every graph $G$.
Indeed, an independent set in $G$
corresponds to a clique in its complement, which has full rank. So,
what changes when we replace the traditional 
\textit{independence number} $\alpha(G)$, with the complement rank?

This new notion originates in our long-lasting failure to
understand the structure of low-rank matrices.
The most ambitious attempt at this mystery is
the famous {\em log-rank} conjecture \cite{lovasz1988lattices}
from communication complexity, which attempts to
characterise low-rank \textit{binary matrices}.
This conjecture (some would call it a {\it problem})
posits, in a form due to Nisan and Wigderson \cite{nisan1995rank},
that any low-rank binary matrix must contain a large monochromatic minor.
An equivalent graph-theoretic formulation of this problem,
due to Lov{\'a}sz and Saks (also in \cite{nisan1995rank}),
asks whether there exists an absolute constant $c$ such that:
\[
    \log \chi(G) \le \mathcal{O} \left( \log^c  \rank(G) \right), 
\]
for every graph $G$. Here $\chi(G)$ is the chromatic number of $G$.

Constructions of Rank-Ramsey graphs
as well as impossibility results
are deeply connected to the log-rank conjecture,
as first suggested in \cite{lee2023around}.
As we show here, any construction of Rank-Ramsey graphs 
yields a separation in the log-rank conjecture,
and conversely, under certain conditions, 
proving the 
impossibility of such graphs may validate the conjecture. We elaborate on
this connection in \Cref{sect:rank_ramsey_and_log_rank}.

The difficulty of characterising low-rank matrices is profound.
Even understanding {\em typical} low-rank matrices is a mystery.
Indeed, culminating 60 years of excellent research,
Tikhomirov \cite{tikhomirov2020singularity} showed
that an $n\times n$ matrix with random
$\pm 1$ entries is singular with probability $(\frac{1}{2}+o_n(1))^n$,
which is clearly tight up to
the little-oh term.
There is a rich literature
of theorems of the same vein, showing 
that full-rank matrices are the rule 
rather than the exception.

Random constructions are
key to the study of Ramsey numbers.
Already over 75 years ago,
Erd\H{o}s \cite{erdos1947some} discovered that 
asymptotically almost all 
graphs are Ramsey.
In contrast, random graphs typically have \textit{logarithmic} independence number,
but \textit{full} complement rank. 
This poses new non-trivial challenges in the 
construction of Rank-Ramsey graphs and in understanding the range of possible values.
We also consider two
of the most interesting and deep constructions
of Ramsey Graphs \cite{alon1994explicit, codenotti2000some},
and both appear of little value in the study of Rank-Ramsey graphs. 

As mentioned, our basic definition hinges on
the observation that the submatrix that
corresponds to an independent set in $G$
has full complement rank. Such an interplay between 
cardinality (here, of independent sets in $G$)
and rank (of the corresponding submatrix in the adjacency matrix of $\overline{G}$) is an
old and fundamental theme in combinatorics. Babai and Frankl exhibit in their book \cite{babai2020linear} 
numerous fascinating examples of proofs
that compare between cardinality and rank. Typically, one proves
a desired lower bound on the cardinality of a 
set by showing that a certain associated matrix
has large rank.
The historically oldest example
known to us of this method is Hanani's proof
\cite{hanani1951number} of the de Bruijn-Erd\H{o}s Theorem.\footnote{
While Hanani's paper dates to 1951, there are
records that it was first written (without
publishing) in the late 1930's.} It is 
interesting to note that in contrast with such
classical proofs, in the study of
the Rank-Ramsey problem, rank bounds cardinality
{\em from above}.

\paragraph{Our Results.}
We first present some constructions of 
Rank-Ramsey graphs. As mentioned,
using the probabilistic method
to this end is not an option, since we lack
a natural distribution over low-rank matrices.
In this view, we opt for \textit{explicit}
constructions of Rank-Ramsey graphs, yet
relying at times on some probabilistic machinery.
We begin with two constructions of Rank-Ramsey graphs 
exhibiting a \textit{polynomial separation} 
between complement rank and order.

We need to introduce some notation first.
We consider
the least complement rank of an $n$-vertex graph whose 
clique number is at most $d$.\footnote{
Throughout the paper we measure $\rank(A_G+I)$ only, 
rather than $\rank(\overline{G})$.
Note that $A_G + I = J - A_{\overline{G}}$, 
where $J$ is the all-ones matrix,
therefore the two quantities differ by at most $1$.
Nevertheless, we stress that even when we use (for convenience) the term ``complement rank'',
we are always referring only to $\rank(A_G + I)$.}
\begin{definition*}
    For every positive $d$, let $\nu_d: \mathbb{N} \to \NN$ be the function
    \[
        \nu_d(n) \eqdef \min_{G} \rank( A_G + I )
    \]
    minimizing over all $n$-vertex graphs $G$ with clique number $\le d$.
\end{definition*}
As a brief illustration, note that
$\nu_d(n) \le \lfloor n/d \rfloor$: Take
the disjoint union of $\lfloor n/d \rfloor$
cliques $K_d$ (and a clique on the remaining
$n \pmod d$ vertices).
Notwithstanding the simplicity of this inequality, it is not easy to beat.
In particular, no graph with fewer than $10$ vertices can accomplish this. \\

With this notation, here is our first result:

\begin{thm}
    \label{thm:kron_construction_1}
    The following bounds hold:
    \begin{enumerate}
        \item $\nu_{41}(n) = \mathcal{O} \left( n^{1 - \frac{1}{10000}} \right)$.
        \item For any sufficiently large $d$, there holds $\nu_d(n) = \mathcal{O} \left( n^{\log_{296} (232)} \right)$, where $\log_{296} (232) \approx0.957 $.
    \end{enumerate}
\end{thm}

The proof of \Cref{thm:kron_construction_1} starts with a base-graph, that
beats the above \textit{trivial bound}.
We then take repeated strong products 
of this graph, to amplify its already low complement rank,
and show that one can find a \textit{large subgraph} of the product,
with no cliques above a certain size.
The search for
a base-graph turns out to be challenging. 
To this end, we turn to strongly-regular graphs (SRGs):
Highly-structured regular graphs with only three distinct eigenvalues, one of which, the Perron eigenvalue, is simple. 

A particularly good base-graph would have both low complement rank \textit{and} small clique number, say, be triangle-free.
However, we encounter two immediate barriers.
Only seven triangle-free strongly regular graphs are known: $C_5$, Petersen, Clebsch,
Hoffman-Singleton, Gewirtz, Mesner, and the Higman-Sims graph.
Whether this list is exhaustive or not is a famous open question (Cf. \cite{biggs1993algebraic,brouwer2012strongly}).
Moreover, another problem arises. We show that triangle-free SRGs must have full complement rank.

To handle these difficulties, we use a ``clique tensoring trick'',
allowing us to produce new graphs with low complement rank from graphs whose spectrum
has certain properties, while retaining the clique number. 
Among the aforementioned list of triangle-free SRGs, a unique good candidate emerges: the Clebsch graph
(and for $K_4$-free, the complement of the Schl\"afli graph).
To our knowledge this is at least the second time the Clebsch graph had appeared in the context of Ramsey theory,
the first being the proof of the multicolour Ramsey number, $R(3,3,3)=17$, due to Greenwood and Gleason \cite{greenwood1955combinatorial}.

Our second result is another construction of a  Rank-Ramsey family.
\begin{thm}
    \label{thm:ae_construction_2}
    For any two constants $c, \varepsilon > 0$ with $c > 2 \left(\frac{2}{3\varepsilon}+1\right)^2$,
    there holds
    $
        \nu_{ c \log n}(n) \le \widetilde{\mathcal{O}} \left(n^{\frac{2}{3} + \varepsilon} \right)
    $.
\end{thm}
The graphs in this second family have complement rank 
near $\mathcal{O}(n^{2/3})$, and logarithmic clique number.
One can take, e.g.,
$\varepsilon=10^{-3}$ and $c=10^6$.

A key component in the proof of \Cref{thm:ae_construction_2} is a matrix-theoretic view of lifts
that we adopt. As we elaborate below (\Cref{subsect:matrix_view_lifts}) {\em lifting}
is a well-established proof technique in the field of communication complexity. However,
we apply it a little differently than usual, and consider the composition of a Boolean function 
$f: \{0,1\}^n \to \{0,1\}$, and a \textit{collection} of binary matrices, not necessarily all identical.
We show a fundamental connection between the expansion of $f$ as a real multilinear polynomial,
and the resulting lifted matrix.
This implies a lifting theorem of-sorts:
The rank of the lifted matrix is determined by the expansion of $f$, and by the ranks
of its constituent matrices.

The Boolean function used in \Cref{thm:ae_construction_2} is the well-known
$\NAE: \{0,1\}^3 \to \{0,1\}$. This function appeared first
in the context of the log-rank conjecture in Nisan and Wigderson's
seminal paper \cite{nisan1995rank}.
The vertices of the lifted graph can be naturally identified with the points of the cube, $[k]^3$.
We show that within this graph, cliques emerge either
(i) From correlations between the lifted matrices, or 
(ii) From certain degenerate subsets of the cube.
The former case is handled by lifting three i.i.d.\ Erd\H{o}s-R\'enyi graphs. 
The latter issue is dealt with by means of a probabilistic argument.
We apply a random construction to find a large subset $S$ of the cube,
such that \textit{every} subset $T \subset S$ of a certain size, is far from ``cube-like''.
That is, $T$ has \textit{some} axis-projection of size \textit{linear} in its cardinality.
Our lifts are unorthodox in two ways.
Firstly, all \textit{gadgets} (i.e., matrices) are distinct, and in fact, uncorrelated.
Secondly, the traditional roles are reversed; the function
has \textit{constant size}, while the gadgets \textit{grow asymptotically}.\\

To present our third result, we require yet another bit of notation. As usual, the Ramsey
number $R(s,t)$ is the least number $n$ such that every $n$-vertex graph has
either an $s$-clique or a $t$-anticlique.
Extending the analogy, we define the ``KRamsey'' numbers.\footnote{We leave it
to the reader's taste and judgement how
to pronounce this new name.}
\begin{definition*}
    $R^k(s,t)$ is the least number $n$ such that every order-$n$ graph $G$ has either an $s$-clique, or
    \[
        \rank(A_G + I) \ge t.
    \]
\end{definition*}
We characterise the KRamsey numbers, when the rank is small.
\begin{thm}
    \label{thm:kramsey_3}
    For $2 \le t \le 5$ and every 
    $s > 1$, there holds 
    $R^k(s, t) = (s-1)(t-1) + 1$.
\end{thm}

Unlike the classical Ramsey numbers, KRamsey numbers are not symmetric in their parameters:
it is much harder to keep the rank low than it is to avoid large cliques.
Indeed, the triangle-free base-graphs used in \Cref{thm:kron_construction_1}, together 
with the characterisation of \Cref{thm:kramsey_3}, imply that $R^k(3,n) > R^k(n,3)$ for every sufficiently large $n$.
Also, it is not hard to see that $(s-1)(t-1) < R^k(s,t) \le R(s,t)$ 
for every $s$ and $t$.
Another stark contrast is, therefore, that while Ramsey numbers 
always strictly exceed
this trivial bound (whenever $s,t > 2$), for KRamsey numbers this is evidently untrue.

The proof of \Cref{thm:kramsey_3} relies on \textit{blowup} of graphs,
where vertices are replaced with anticliques and edges by complete bipartite graphs.
We remark that for every rank $r$, there exists a finite list 
of graphs $\mathcal{G}_r$
such that every connected graph of rank $r$ is a blowup of a graph in $\mathcal{G}_r$.
These lists appear in our proof, in particular those of ranks $r=4$ \cite{chang2011characterization}
and $r=5$ \cite{chang2012characterization}, found by Chang, Huang and Ye. \\

Next we consider
triangle-free Rank-Ramsey graphs.
Following a long line of research \cite{kim1995ramsey}, the Ramsey number $R(3,t)$ is now known up to a small multiplicative factor
\cite{fiz2020triangle}.
In this view, in order to improve the bounds on $R^k(3, t)$, we must resort to graph parameters
other than independence number. Recall that
an \textit{orthonormal representation} of a graph $G$ is
an assignment of unit vectors to its vertices such that vectors
of non-adjacent vertices are orthogonal. Let $M$ be
the Gram matrix of these vectors.
The least dimension of such a representation of $G$ is denoted $\msr(G)$. 
The Lov{\'a}sz number $\vartheta(G)$ \cite{lovasz1979shannon} 
is another well-known graph parameter, related to orthonormal representations, 
whose definition we omit here for brevity. An important relation between
these two quantities and with Ramsey graphs
is the well-known ``sandwich theorem''
(see \cite{knuth1993sandwich}), which states that for every graph $G$ there holds
\[
    \alpha(G) \le \vartheta(G) \le \msr(G) \le \chi(\overline{G}).
\]

The matrices $M$ and $A_G + I$ agree both on the main diagonal, and on the entries of
non-edges of $G$. 
What other properties do they share?
We prove the following. 

\begin{thm}
    \label{thm:lb_thm_4}
    For infinitely many $n>1$, there exist $n$-vertex triangle-free graphs 
    $G_1$ and $G_2$, with
    \begin{enumerate}
        \item $\vartheta(G_1) = \Theta(n^{2/3})$ and $\rank(A_{G_1} + I) = n$.
        \item $\msr(G_2) \ge n/2$ and $\rank(A_{G_2} + I) = (3/8 + o(1))n$.
    \end{enumerate}
\end{thm}

The first result comes from an
explicit construction of triangle-free Ramsey graphs, due to 
Alon \cite{alon1994explicit}. These 
$n$-vertex graphs have $\vartheta(G)=\Theta(n^{2/3})$, so
by the Sandwich Theorem, these are Ramsey graphs.
The bound on $\vartheta(G)$ is best possible, so they are in fact
optimal $\vartheta$-Ramsey graphs.\footnote{
By $\vartheta$-Ramsey we mean graphs with small clique number and low 
\textit{Lov{\'a}sz number}, rather
than \textit{independence number}.} However, as
we show, these graphs are as bad as possible for the Rank-Ramsey property 
in that they satisfy $\rank(A_G + I) = n$.
These are Cayley graphs of the Abelian group $\mathbb{Z}_N$, 
and so our proof is Fourier-analytic.
As for the latter result, concerning $\msr(G)$, 
the lower bound follows from a result of Deaett \cite{deaett2011minimum},
and our converse follows from the Clebsch-derived family of graphs.

The relations between minimum semi-definite rank, Lov{\'a}sz number, and complement rank
are still not sufficiently well-understood.
We observe several similarities between the Lov{\'a}sz number and the rank of
\mbox{$A_G+I$}:
they are derived from closely related matrices, both are
multiplicative in the strong graph product, and both
bound the Shannon capacity from above.
The possibility that $\rank(A_G + I) \ge \Omega(\vartheta(G))$ is especially intriguing,
and we give some supporting evidence in this direction.
For example, in every graph with fewer than $10$ vertices
the Lov{\'a}sz number is at least as big as the complement rank.
We know of no counterexample.
We remark that, if this holds true, then Alon and Kahale’s \cite{alon1998approximating}
extension of Kashin and Konyagin’s bound \cite{kashin1981systems}
on the Lov{\'a}sz number of graphs with bounded independence number
implies a polynomial improvement in our lower-bounds on $\nu_d(n)$, for every positive constant $d$.
In \Cref{sect:hoffman_cvetkovic_lovasz_ramsey} we draw connections between this problem
and other well-known questions in algebraic
graph theory, related to the Hoffman and Cvetkovi\'c bounds on the independence number. \\

More key results in Ramsey theory,
and regarding the log-rank conjecture,
can be observed through the lens of the Rank-Ramsey problem. 
For instance, there are additional interesting \textit{explicit} constructions of triangle-free Ramsey graphs,
such as Frankl and Wilson's set intersections \cite{frankl1981intersection}, or Chung's fibrations \cite{chung1993note}.
They too can be shown not to be Rank-Ramsey. 
In Appendix \ref{sect:2i_construction} we sketch a shortened version of a noteworthy elementary construction,
due to Codenotti, Pudl{\'a}k and Resta \cite{codenotti2000some}, 
yielding triangle-free Ramsey graphs with $\rank(A_G - 2I) = \mathcal{O}(n^{3/4})$.
These graphs appear close in spirit to our Rank-Ramsey problem\footnote{
Of course, $\rank(A - 2I) = o(n)$ implies $\rank(A_G + I) = (1-o(1))n$, i.e., these too are far from Rank-Ramsey.},
with one key difference,
which we believe to be central: $A_G - 2I$ is not binary.
This is crucially important with regards to the log-rank conjecture,
where it is well known how to construct
non-binary matrices with \textit{constant rank} and \textit{full partition number}. 

We also give a brief analysis of the Nisan-Wigderson construction \cite{nisan1995rank}
of matrices exhibiting a gap in the log-rank problem.
We show that these matrices have large monochromatic principal minors, 
and therefore yield poor Rank-Ramsey graphs as-is.
We stress that this is no contradiction:
exhibiting a log-rank separation does not preclude a matrix from
\textit{having} a large monochromatic rectangle.\footnote{
However, due to the relation between log-rank and Rank-Ramsey, a graph family \textit{related} to said matrices
must contain large Rank-Ramsey \textit{subgraphs}.}

\subsection{Paper Organization}

\paragraph{Log-Rank.} To start we introduce
two central topics that permeate this paper:
Matrix-theoretic lifts (\Cref{subsect:matrix_view_lifts}) and Rank-Ramsey graphs 
(\Cref{subsect:rank_ramsey}).
In \Cref{subsect:rank_ramsey} we also introduce the Rank-Ramsey numbers, and 
compute them precisely for graphs of low complement rank.
In \Cref{sect:rank_ramsey_and_log_rank}, we draw connections between the Rank-Ramsey problem
and the log-rank conjecture in communication complexity.

\vspace{-0.2cm}\paragraph{Constructions.} \Cref{sect:kron_powers} and \Cref{sect:ae_lift} are dedicated to constructions of Rank-Ramsey families.
The first utilizes minors of Kronecker powers of families derived from strongly-regular graphs with particular properties,
and yields a \textit{polynomial separation} between complement rank and order for graphs with \textit{constant} clique number.
The latter relies on matrix-theoretic lifts of Erd\H{o}s-R\'enyi graphs with the well-known Boolean function $\NAE$,
and produces graphs whose clique number is \textit{logarithmic}.

\vspace{-0.2cm}\paragraph{Analysis and Bounds.} \Cref{sect:lb_sect} revolves 
around lower bounds
on the complement rank of triangle-free Rank-Ramsey graphs.
There, we draw connections to other known graph parameters, and also analyse the two best known explicit constructions of triangle-free Ramsey graphs, due to Alon \cite{alon1994explicit},
and Codenotti, Pudl{\'a}k and Resta \cite{codenotti2000some}, showing they are far from Rank-Ramsey.
In \Cref{sect:nw_analysis} we consider the Nisan-Wigderson construction \cite{nisan1995rank}, which first
exhibited a polynomial separation for the log-rank conjecture, from a Rank-Ramsey perspective.
We find large monochromatic principal minors in said matrices, implying that they induce poor Rank-Ramsey graphs.

\section{Preliminaries}
\label{sect:perlims}

\paragraph{Graphs.}

By default, all graphs in this paper are undirected and simple.
We occasionally also consider directed graphs, 
as well as graphs with self-loops.
The \textit{order} of a graph, $v(G)$, is the number of vertices,
and its \textit{size}, $e(G)$, is the number of edges.
If two vertices $u,v \in V(G)$ are adjacent, we denote this by
$u \sim_G v$ and resp.\ $u \to_G v$ (in the undirected, resp.\ 
directed case). The subscript $G$ is omitted
when the graph is clear from the context. 

Associated with every
order-$n$ graph $G = ([n], E)$ is its \textit{adjacency matrix}:
\[
    A_G \in M_n(\RR), \text{ where } \forall i,j \in [n]: (A_G)_{i,j} = \mathbbm{1} \left\{ i \sim j \right\} \ \ \text{(respectively, $i \to j$, if $G$ is directed)}
\]

The \textit{spectrum} of a graph, denoted $\spec(G)$, is the multiset of eigenvalues corresponding to its adjacency matrix.
If $G$ is undirected, its spectrum is real (since $A_G$ is symmetric).
The multiplicity of an eigenvalue $\lambda$ in a
(diagonalizable) matrix $A$ is denoted by $\mu_A(\lambda)$. 
The \textit{rank} of a graph is the real rank of its adjacency.

The subgraph induced by a set of vertices $S \subseteq V(G)$ is denoted $G[S]$.
The \textit{clique number}, the \textit{independence number}, and \textit{chromatic number}
of a simple graph $G$ are denoted $\omega(G)$, $\alpha(G)$ and $\chi(G)$, respectively.
We denote the complement of a graph $G$, by $\overline{G}$. The adjacency matrix
of $\overline{G}$ is $J - (A_G + I)$.
Observe that
\[
    \big| \rank(\overline{G}) - \rank(A_G + I) \big| \le 1,
\]
therefore, where an additive error of $\pm 1$ is insignificant, we sometimes
refer to $\rank(A_G+I)$ as the ``complement rank'' of $G$.\\
A {\em blowup} of a graph $G$ is 
a graph attained by replacing each vertex 
by a nonempty anticlique 
and each edge by a complete bipartite graph.
If $w(i)$ is the size of the anticlique that
replaces vertex $i$, we denote the resulting
graph by $G^w$. For matrices, the blowup $A^w$ of a matrix $A$ replaces each entry $A_{i,j}$
with a $(w(i) \times w(j))$-block of entries $A_{i,j}$ 
(so $A_{G^w} = A_G^w$).
Importantly, $\rank(A^w) = \rank(A)$ for any $A$ and $w$.\\
Two vertices are called \textit{twins} if they are non-adjacent and
have the same set of neighbours.
Discarding twins from a graph is called a \textit{reduction},
and is the inverse operation to blowup.
Both blowup and reduction affect neither the chromatic number, nor the rank of a graph.

\paragraph{Boolean Functions.}
It is well known that any Boolean function $f: \{0,1\}^n \to \{0,1\}$
is uniquely representable as a multilinear polynomial over the Reals, viz.,
\[
    f(x) = \sum_{S \subseteq [n]} a_S \prod_{i \in S} x_i \in \RR[x_1, \dots, x_n]
\]

As usual, the \textit{support} of $f$ is the index set of
the non-zero coefficients in this representation, denoted
$\mon(f) \eqdef \{ S : \emptyset \ne S \subseteq [n], a_S \ne 0\}$.
We denote the size of $f$'s support by $\spar(f) \eqdef |\mon(f)|$,
and the \textit{degree} of $f$ is $\deg(f) \eqdef \max \{|S| : S \in \mon(p)\}$.
By convention, the constant functions 
$f \equiv 0$ and $f \equiv 1$ are defined to have degree $0$.

\paragraph{Matrices.} As usual, the identity matrix is denoted $I$, and the all-ones matrix $J$.
The trace of a square matrix is the sum of entries on its main diagonal.
Aside from the regular matrix product, we use the
Kronecker and Hadamard products which we now recall.
Let $A,C \in M_{m \times n}(\RR)$, and $B \in M_{k \times l}(\RR)$ be three matrices.
The Kronecker product $A \otimes B \in M_{mk \times nl}(\RR)$ and Hadamard (element-wise) product 
$A \odot C \in M_{m \times n}(\RR)$ are defined by:
\[
    (A \otimes B)_{i,j,x,y} \eqdef A_{i,j} B_{x,y} , \text{ and } (A \odot C)_{i,j} \eqdef A_{i,j} C_{i,j}   
\]

It is well known that rank is multiplicative under the Kronecker product, and sub-multiplicative under the Hadamard product.
We also require the standard \textit{Kronecker} and the \textit{Strong} graph products.
\begin{definition}
    \label{defn:graph_kron}
    The Kronecker Product, $G \otimes H$, and the Strong Product, $G \boxtimes H$, of two simple graphs $G=([n],E_G)$ and $H=
    ([k],E_H)$, are graphs with vertex set $[n] \times [k]$, where
    \begin{align*}
        (i,a) \underset{G \otimes H}{\widesim{}} (j,b) &\iff (i \underset{G}{\widesim{}} j) \land (a \underset{H}{\widesim{}} b) \\
        \text{and} \quad\quad\quad\quad\quad\quad\quad\quad\quad\quad\quad\quad\quad\quad\quad& \\ 
        (i,a) \underset{G \boxtimes H}{\widesim{}} (j,b) &\iff (i = j \land a \underset{H}{\widesim{}} b) \lor (a = b \land i \underset{G}{\widesim{}} j) \lor \left((i \underset{G}{\widesim{}} j) \land (a \underset{H}{\widesim{}} b) \right)
    \end{align*}
    In other words, their adjacency matrices are $A_G \otimes A_H$ and $A_G \otimes A_H + A_G \otimes I_k + I_n \otimes A_H$, respectively. 
\end{definition}

Unlike the usual notion of matrix minors, \textit{Generalised Minors}
allow repeated indices.

\begin{definition}
    \label{defn:generalized_minors}
    Let $A \in M_{k}(\RR)$ be a matrix.
    The Generalised Minor of $A$ corresponding to the indices $(i_1, \dots, i_d) \in [k]^d$
    is the matrix $A[i_1, \dots, i_d] \in M_{d}(\RR)$ defined by
    \[
         \left(A[i_1, \dots, i_d]\right)_{s,t} \eqdef A_{i_s, i_t}, 
    \]
    for every $s,t \in [d]$.
\end{definition}

Note that the \textit{rank} of a generalised minor cannot exceed the rank of its ``parent'' matrix.

\begin{proposition}
    \label{prop:rank_of_gen_minor}
    If $A \in M_{k}(\RR)$ be a matrix and $(i_1, \dots, i_d) \in [k]^d$
    are indices, then \[\rank(A[i_1, \dots, i_d]) \le \rank(A)\]
\end{proposition}
\begin{proof}
Let $P$ be a $d \times k$
binary matrix with a $1$ in position $(x,y)$ iff $i_x = y$.
The claim follows from the observation that $A[i_1, \dots, i_d]$ coincides with
$P A P^T$, and the fact that $\rank(P A P^T)\le \rank(A)$.\end{proof}

\paragraph{Communication Complexity.}
We follow the standard notation and terminology 
in communication complexity, as introduced in the excellent book
\cite{kushilevitz1997communication}.
Let $A \in M_{m \times n}(\RR)$ be a binary matrix.
A \textit{combinatorial rectangle} of $A$ is a minor with
row set $R \subseteq [m]$ and column set $C \subseteq [n]$.
For a bit $b \in \{0,1\}$, a $b$-\textit{rectangle cover} of $A$ is a set 
of $b$-monochromatic combinatorial rectangles whose
union comprises all $b$-entries in $A$.
The $b$-cover number of $A$, and its $b$-nondeterministic communication,
are defined as follows:\footnote{All logarithms in this paper are \textit{binary} (base 2), unless stated otherwise.}
\[
    \chi^b(A) \eqdef \min_{\mathcal{R}} |\mathcal{R}|, \text{ and } \mathrm{N^b}(A) \eqdef \lceil \log \chi^{b}(A) \rceil
\]
where the minimum is taken over all $b$-rectangle covers of $A$.

As usual, the \textit{deterministic} communication complexity of $\DCC(A)$,
is the least cost of a deterministic communication protocol computing $A$
(where the cost of a protocol is the maximum, over all inputs, of the number of bits transmitted).
\subsection{Matrix Representation of Lifts}
\label{subsect:matrix_view_lifts}

\textit{Lifting} is a powerful
technique in the study of communication 
complexity, first introduced by Raz and McKenzie 
\cite{raz1997separation}. An application of this method
starts with two Boolean functions, \mbox{$f: \{0,1\}^n \to \{0,1\}$} and 
$g: \{0,1\}^b \times \{0,1\}^b \to \{0,1\}$. Typically $b$ is much
smaller than $n$, and $g$ is commonly called a \textit{gadget}.
The corresponding lift is the composition 
$f \circ g^n: (\{0,1\}^{b})^n \times (\{0,1\}^{b})^n \to \{0,1\}$:
\[
    (f \circ g^n) \big( (x_1, y_1), \dots, (x_n, y_n) \big) \eqdef f \big( g(x_1, y_1), \dots, g(x_n, y_n) \big),
\]
where $x_i, y_j\in \{0,1\}^b$.

A lifting theorem establishes a relation between
the \textit{communication complexity} of 
$f \circ g^n$, the \textit{query complexity} of $f$ and some 
property of $g$. This is done
for a particular choice of query and communication models. The logic
behind such theorems is this: Generally speaking it is hard to prove
lower bounds for a communication model. Lifting theorems accomplish this task
by relying on a lower bound in a query model, which is easier to come by.
There are numerous examples of such theorems, c.f. \cite{raz1997separation, goos2015deterministic, goos2016rectangles, chattopadhyay2019simulation, goos2017query}.

\textit{We take a matrix-theoretic view of lifts}.
Rather than compose \textit{two} Boolean functions, we compose 
a Boolean function with a \textit{collection} of binary matrices,

\begin{definition}
    \label{defn:f_lift}
    Let $A_1, \dots, A_n \in M_{m \times k}(\RR)$ be binary matrices and let $f: \{0,1\}^n \to \{0,1\}$ be a Boolean function. The $f$-lift of $A_1, \dots, A_n$, denoted $f \big( A_1, \dots, A_n \big)$, is the $m^n \times k^n$ Boolean matrix, 
    \[
        f \big( A_1, \dots, A_n \big)_{i, j} \eqdef f\big( (A_1)_{i_1, j_1}, \dots, (A_n)_{i_n, j_n} \big), 
    \]
    for $i = (i_1,\ldots,i_n) \in [m]^n$ and $j = (j_1,\ldots,j_n) \in [k]^n$.
\end{definition}

This notion of lifting binary matrices with a Boolean function $f: \{0,1\}^n \to \{0,1\}$
is intimately related to the multilinear representation of $f$
over the Reals, and to the Kronecker product of matrices.

\begin{proposition}
    \label{prop:matrix_repr_of_lift}
    Let $A_1, \dots, A_n \in M_{m \times k}(\RR)$ be binary matrices and let $f: \{0,1\}^n \to \{0,1\}$
    be a Boolean function whose multilinear expansion is
    $f(x) = \sum_{S \subseteq[n]} a_S \cdot \prod_{l \in S} x_l \in \RR[x_1, \dots, x_n]$.
    Then, the $f$-lift of $A_1, \dots, A_n$ can be written as:
    \[
        f \big( A_1, \dots, A_n \big) = \sum_{S \subseteq [n]} a_S \cdot \underset{l \in [n]}{\Motimes} Z^S_l 
    \]
    where $J_{m \times k}$ is the $(m \times k)$ all-ones matrix and $Z^S_l = \begin{cases} A_l & l \in S \\ J_{m \times k} & l \notin S \end{cases}$
\end{proposition}
\begin{proof}
    By the definition of the Kronecker product, $\forall (i,j) \in [m]^n \times [k]^n$ we have:
    \begin{align*}
        \left(\sum_{S \subseteq [n]} a_S \cdot \underset{l \in [n]}{\Motimes} \begin{cases} A_l & l \in S \\ J_{m \times k} & l \notin S \end{cases} \right)_{i,j} &= \sum_{S \subseteq [n]} a_S \cdot \prod_{l \in [n]} \left(\begin{cases} A_l & l \in S \\ J & l \notin S \end{cases} \right)_{i_l, j_l} \\
        &= \sum_{S \subseteq [n]} a_S \cdot \prod_{l \in S} (A_l)_{i_l, j_l} \\
        &= f \big( (A_1)_{i_1, j_1}, \dots, (A_n)_{i_n, j_n} \big) = f \big( A_1, \dots, A_n \big)_{i,j} \qedhere
    \end{align*}
\end{proof}

This relationship implies the following powerful bound on the rank of a lifted matrix. 

\begin{lemma}
    \label{lem:rank_of_lift}
    Let $f: \{0,1\}^n \to \{0,1\}$ be a Boolean function and let $A_1, \dots, A_n \in M_{m \times k}(\RR)$ be binary matrices. Then,
    \[
        \rank \left( f(A_1, \dots, A_n) \right) \le \sum_{S \in \mon(f)} \prod_{i \in S} \rank(A_i)
    \]
\end{lemma}
\begin{proof}
    Let $\sum_{S \in \mon(f)} a_S \prod_{i \in S} x_i$ be the multilinear representation of $f$ over the Reals. By \Cref{defn:f_lift},
    \begin{align*}
        f(A_1, \dots, A_n) = \sum_{S \in \mon(f)} a_S \bigotimes_{i \in [n]} \begin{cases} A_i & i \in S \\ J_{m \times k} & i \notin S \end{cases}
    \end{align*}

    The proof now follows by recalling that $\rank(J) = 1$, and that rank is subadditive, and multiplicative under the Kronecker product.
\end{proof}

\subsection{Rank-Ramsey Graphs}
\label{subsect:rank_ramsey}

The Ramsey number $R(s,t)$ is the smallest $n$ such
that every graph $G$ of order $n$ contains either an \mbox{$s$-clique}
or an $t$-anticlique. In other words, either $\omega(G)\ge s$ or
$\alpha(G)\ge t$. Observe that $\rank(A_G + I)\ge \alpha(G)$,
since the minor corresponding to an anticlique in $G$ is
all-zeros in $A_G$, and an identity submatrix in $A_G + I$.
This suggests a search for graphs where both the clique number $\omega(G)$
and $\rank(A_G + I)$ are small.
We call graphs with such properties \textit{Rank-Ramsey}.
To proceed, we introduce some notation.

\begin{definition}
    For every $d \ge 1$, let $\nu_d: \mathbb{N} \to \NN$ be the function
    \[
        \nu_d(n) \eqdef \min_{G} \rank( A_G + I )
    \]
    minimizing over all order-$n$ graphs $G$ with $\omega(G) \le d$.
\end{definition}

A similar notion for directed graphs is of interest as well,

\begin{definition}
    For every $d \ge 1$, let $\eta_d: \mathbb{N} \to \NN$ be the function
    \[
        \eta_d(n) \eqdef \min_{A \in M_n(\RR)} \rank( A + I )
    \]
    minimizing over all $n \times n$ binary matrices $A$ with zeros on the main diagonal, 
    and such that $(A+I)$ has no $J_{d+1}$ principal minor.
\end{definition}

There is a simple relation between Ramsey numbers and the function $\nu_d$: 
\begin{proposition}
    \label{prop:nu_vs_R}
    For any two positive integers $d$ and $n$, there holds $n < R(d+1, \nu_d(n) + 1)$. 
\end{proposition}
\begin{proof}
    If $G$ is a graph attaining $\nu_d(n)$, then
    by definition $\omega(G) < d+1$. Also, as mentioned, $\alpha(G) < \rank(A_G + I) + 1 = \nu_d(n) + 1$. Therefore, the order of $G$
    is smaller than $R(d+1, \nu_d(n) + 1)$, as claimed. 
\end{proof}
\begin{corollary}[\cite{lee2023around}]
    \label{cor:nu_vs_ramsey}
    For any two positive integers $n > d$, there holds
    \[
        \nu_d(n) = \Omega \left( n^{1/d} \cdot \left( \frac{\log n}{d}\right)^{\tfrac{d-1}{d}} \right)
    \]
\end{corollary}
\begin{proof}
    Apply the classical bounds on Ramsey numbers, due to Ajtai, Koml{\'o}s and Szemer{\'e}di \cite{ajtai1980note}. 
\end{proof}

The following lemma aggregates several useful properties of our two quantities.

\begin{lemma}
    \label{lem:props_nu_eta}
    For every natural number $d$, it holds that:
    \begin{enumerate}
        \item
        Both quantities $\nu_{d}(n)$ and $\eta_{d}(n)$ are non-increasing in $d$.
        \item $\eta_d(n) \le \nu_d (n) \le \eta_d(n)^2$.
        \item Both $\nu_d$ and $\eta_d$ are submultiplicative: $\nu_d(k n) \le k \cdot \nu_d(n)$, and $\eta_d(k n) \le k \cdot \eta_d(n)$.
    \end{enumerate}
    and moreover, for every $d_1, d_2$ and $n_1, n_2$, we have:
    \begin{enumerate}[resume]
        \item $\nu_{d_1 d_2}(n_1 n_2) \le \nu_{d_1}(n_1) \nu_{d_2}(n_2)$, and similarly $\eta_{d_1 d_2}(n_1 n_2) \le \eta_{d_1}(n_1) \eta_{d_2}(n_2)$.
    \end{enumerate}
\end{lemma}
\begin{proof} 
    We prove each property in turn,
    \begin{enumerate}
        \item Obvious: If $b>a$, then a graph with no $a$-clique has no $b$-clique either. 
        \item The lower bound is obvious. For the upper bound, let $A$ be the matrix attaining $\eta_d(n)$ and consider the Hadamard product $(A+I) \odot (A+I)^T$.
        \item Let $G$ be a graph attaining $\nu_d(n)$,
        and let $H$ be the disjoint union of $k$ copies of $G$. Clearly $\omega(H) = \omega(G)$ and $\rank(A_H + I) = k \cdot \rank(A_G + I)$. The proof for $\eta_d$ is identical.
        \item Let $G_1$ and $G_2$ be graphs attaining the minimum for $\nu_{d_1}(n_1)$ and $\nu_{d_2}(n_2)$, respectively.
        Let $\Gamma$ be the graph whose adjacency matrix is $A_\Gamma = (A_{G_1} + I) \otimes (A_{G_2} + I) - I$.
        By multiplicativity,
        \[
            \rank( A_\Gamma + I) = \rank \left( (A_{G_1} + I) \otimes (A_{G_2} + I) \right) = \rank(A_{G_1} + I) \cdot \rank(A_{G_2} + I)
        \]
        
        But for any clique $S$ in $\Gamma$,
        we have $|S| \le \omega(G_1) \cdot \omega(G_2)$, since its projections on
        either coordinate are cliques in $G_1$ and $G_2$, respectively. The same proof applies to $\eta$.
        $\qedhere$ 
    \end{enumerate}
\end{proof}

\subsubsection{Rank-Ramsey Numbers}

Maintaining the analogy with Ramsey numbers, we define the \textit{Rank-Ramsey numbers},
\begin{definition}
    $R^k(s,t)$ is the smallest integer $N$ such that
    for every graph of order $N$, 
    \[ \omega(G) \ge s \text{ or } \rank(A_G + I) \ge t \]
\end{definition}

Clearly, $(s-1)(t-1) < R^k(s,t) \le R(s,t)$. 
The lower bound follows by taking the
disjoint union of $(t-1)$ copies of $K_{s-1}$. 
For small numbers $s$ and $t$, Rank-Ramsey numbers are
\textit{strictly smaller} than Ramsey numbers.
In fact they match the trivial bound $R^k(s,t) = (s-1)(t-1) + 1$
for all $2 \le t \le 5$ and any $1<s$. This inequality fails for substantially
larger values of $t$. For example, we show later on 
(\Cref{cor:nu_2_3_small_bounds}) 
that $R^k(3, 6l+11) > 16l$ for every $l > 2$.
For $l=6$ this yields $R^k(3, 47) > 96 > 93 = 2 \cdot 46 + 1$.
The following theorem shows, in contrast, that when the complement rank is low,
the trivial lower bound is \textit{always} attained.
Consequently, $R^k(3,n) > R^k(n,3)$, for every sufficiently large $n$,
and in particular, unlike Ramsey numbers, Rank-Ramsey numbers are not symmetric in their 
parameters. To wit, it is much harder to keep the rank low than to avoid large cliques.

\begin{theorem}
    \label{thm:bound_rank_ramsey}
    For $2 \le t \le 5$ and every 
    $s > 1$, there holds 
    $R^k(s, t) = (s-1)(t-1) + 1$.
\end{theorem}

It is not hard to show that a simple connected graph $G$ has rank $2$ iff it is
a complete bipartite graph, i.e., a \textit{blowup} of an edge. Likewise,
$\rank(A_G) = 3$ if and only if $G$
is a blowup of a triangle.
Chang, Huang and Ye \cite{chang2011characterization, chang2012characterization},
showed that every connected graph of 
$\rank$ $4$ or $5$, is the
blowup of a graph from an explicit
finite list of graphs, $\mathcal{G}_4$ and $\mathcal{G}_5$.
This is, in fact, true in general. 
Indeed, if a graph has a pair of twin vertices
and we remove one of the two, the rank remain unchanged (as does connectivity). 
Such reductions can be applied repeatedly.
Also, note that such a reduction is the {\em inverse of a blowup}. Therefore,

\begin{proposition}
For every positive integer $r$, there is a finite list 
of graphs $\mathcal{G}_r$ such that every
connected graph of rank $r$ is a blowup of a member of $\mathcal{G}_r$. 
\end{proposition}
\begin{proof}
As shown by Kotlov and Lov{\'a}sz \cite{kotlov1996rank},
a twin-free graph of rank $r$
has at most $\mathcal{O}(2^{r/2})$
vertices. The claim now follows with
\[
    \mathcal{G}_r = \left\{\ \Gamma : \rank(A_\Gamma) = r \text{ and } \Gamma \text{ is twin-free }\right\}. \qedhere
\]
\end{proof}

An adaptation of the results on $\mathcal{G}_4$ and $\mathcal{G}_5$ 
leads us to the following lemma.

\begin{lemma}
    \label{lem:characterisation_low_rank_a_plus_i}
    Let $G$ be a simple graph whose complement is connected. Then,
    \begin{enumerate}
        \item $\rank(A_G + I) = 4$ if and only if $\overline{G}$ is a blowup of a graph in $
            \left\{ \tikz \graph[n=4]{1,2,3,4;1--2,2--3,3--4};,\ 
            \tikz \graph[n=5]{1,2,3,4,5;1--2,2--3,3--3,4--1,2--5,3--5};,\ 
            \tikz \graph[n=6]{1,2,3,4,5,6;1--2,2--3,3--4,4--1,1--5,4--5,2--6,3--6,5--6};,\ 
            \tikz \graph[n=4]{1,2,3,4;1--2,1--3,1--4,2--3,2--4,3--4};,\ 
            \tikz \graph[n=4]{1,2,3,4;1--2,1--3,2--3,3--4};
             \right\}$.
        \item $\rank(A_G + I) = r$ for $r \in \{2,3\}$ if and only if $\overline{G}$ is a blowup of an $r$-clique.
    \end{enumerate}
\end{lemma}
\begin{proof}
    Let $G$ be an $n$-vertex graph with $\rank(A_G + I) = r$ where $r \in \{2, 3, 4\}$, and $\overline{G}$ connected. Then,
    \[
        l \eqdef \rank(A_{\overline{G}}) = \rank(J - (A_G + I) ) \in \left\{r-1, r, r+1\right\}
    \]
    So, by \cite{chang2011characterization, chang2012characterization}, 
    $\overline{G} = H^w$ is a blowup of a 
    graph $H \in \mathcal{G}_l$ with some weights $w$.
    But then,
    \begin{align*}
        r = \rank(A_G + I_n) &= \rank(J_n - J_n + A_G + I_n) \\
        &= \rank(J_n - A_{H^w}) \\
        &= \rank(J_n - A_H^w) \\
        &= \rank( (J_{v(H)}-A_H)^w ) = \rank( J_{v(H)} - A_H )
    \end{align*}
    therefore $H\in\mathcal{G}_{r-1} \sqcup \mathcal{G}_r \sqcup \mathcal{G}_{r+1}$
    and is such that
    $\rank(J - A_H) = r $. This
    yields the above lists.
\end{proof}
We also require the following lemma.
\begin{lemma}
    \label{lem:bound_comp_disj_clique_blowup}
    Let $G$ be an $n$-vertex graph whose complement is a disjoint union of blowups of cliques and isolated vertices.
    Then, $\rank(A_G + I) \cdot \omega(G) \ge n$.
\end{lemma}
\begin{proof}
    Let $k$ be the number of connected components in $\overline{G}$, say
    \[
        \overline{G} = \overline{G}_1 \sqcup \dots \sqcup \overline{G}_k, \text{ where } \overline{G}_i = K_{l_i}^{w_i} \text{ for every $1 \le i \le k$ }
    \]
    (where here, with slight abuse of notation, isolated vertices are denoted by $K_1$). The clique number is determined by the blowup weights,
    \[
        \omega(G) = \alpha(\overline{G}) = \sum_{i=1}^k \alpha(\overline{G}_i) = \sum_{i=1}^k \norm{w_i}_\infty,
    \]
    and conversely, denoting $l^\star \eqdef \max \left\{ l_i : 1 \le i \le k \right\}$ and $L \eqdef l_1 + \dots + l_k$,
    and letting $w$ be the concatenation of the weights for each of the components,
    we may bound the rank as follows
    \begin{align*}
        \rank(A_G + I_n) = \rank(J_n - A_{\overline{G}})
        &= \rank(J_L^w - A_{K_{l_1}^{w_1} \sqcup \dots \sqcup K_{l_k}^{w_k}}) \\
        &= \rank(J_L - A_{K_{l_1} \sqcup \dots \sqcup K_{l_k}}) \\
        &\ge \rank(J_{l^\star} - A_{K_{l^\star}}) = \rank(I_{l^\star}) = l^\star
    \end{align*}
    where the last inequality follows by taking a principal minor. The claim now follows, as 
    \[
        n = \sum_{i=1}^k \langle w_i, \mathbbm{1} \rangle \le \sum_{i=1}^k l_i \norm{w_i}_\infty \le l^\star \sum_{i=1}^k \norm{w_i}_\infty
        \le \rank(A_G + I) \cdot \omega(G) \qedhere
    \]
\end{proof}
The proof of \Cref{thm:bound_rank_ramsey} now follows.
\begin{proof}
    Let $G$ be an $n$-vertex graph with $n > 1$ and $\rank(A_G + I) = t-1$, and write
    \[
        \overline{G} = \overline{G}_1 \sqcup \dots \sqcup \overline{G}_k, \text{ and } r_i \eqdef \rank(A_{\overline{G}_i}) \text{ for every $1 \le i \le k$ }
    \]
    where $k \ge 1$ is the number of components in $\overline{G}$.
    Clearly, $r_1 + \dots + r_k = \rank(A_{\overline{G}}) = \rank(J - (A_G + I)) \le t$.
    The cases $t = \{2,3,4\}$ are thus handled, since if $k=1$ then
    $\overline{G}$ is a blowup of a clique (recall \Cref{lem:characterisation_low_rank_a_plus_i}),
    and if $k > 1$ each \textit{component} is an isolated vertex or blowup of a clique,
    which is covered by \Cref{lem:bound_comp_disj_clique_blowup}.
    
    It remains to handle $t=5$, in the cases where not all 
    the non-trivial components of $\overline{G}$ are blowups of cliques.
    Let $\mathcal{H}$ be the list of graphs in 
    \Cref{lem:characterisation_low_rank_a_plus_i}, excluding $K_4$. 
    Only one possibility still remains to be dealt with: 
    $\overline{G}$ is a blowup $H^w$ of some graph $H \in \mathcal{H}$,
    together with (possibly) some isolated vertices. We can rule out the existence of isolated vertices, since then:
    \[
        \rank(A_G + I) = \rank(J - A_{\overline{G}}) = \rank \left( \begin{array}{c|c}
            J - A_{H^w} & J \\
            \hline
            J & J
        \end{array} \right)
        = \rank \left( \begin{array}{c|c}
            J - A_{H} & \mathbbm{1}^T \\
            \hline
            \mathbbm{1} & 1
        \end{array}  \right)
    \]
    and for all $H \in \mathcal{H}$, the latter rank is $5 > \rank(A_G+I) = t-1 = 4$, a contradiction.

    Finally we deal with the case $\overline{G} = H^w$.
    Maximal independent sets in $H^w$ correspond to blowups of anticliques in $H$.
    Therefore, the following LP bounds the independence number of blowups of $H$:
    \begin{equation*}
        \begin{array}{ll@{}ll}
        \text{minimize}  & a &\\
        \text{subject to}& x \ge 0& \\
                         & \langle \mathbbm{1}, x \rangle = 1 & \\
                         & \langle \mathbbm{1}_S,  x \rangle \le a &\ \ \  \forall S \in I(H)
        \end{array}
    \end{equation*}
    where $I(H)$ is the set of all independent sets of $H$. Let $\alpha^{\star}_H$ be the optimum of the above program.
    The least optima of this LP over the base-graphs in $\mathcal{H}$ is exactly $1/3$, by direct computation. Therefore,
    \[
        \frac{\omega(G)}{n} = \frac{\alpha(H^w)}{\langle \mathbbm{1}, w \rangle} \ge \alpha^{\star}_H \ge \frac{1}{3} = \frac{1}{\rank(A_G + I) - 1} 
    \]
    and the proof follows by re-arranging.
\end{proof}

\subsubsection{Warmup: Simple Constructions of Rank-Ramsey Graphs}

As \Cref{prop:nu_vs_R} shows,
\textit{any} explicit (constructive) sublinear bound on 
$\nu_d(n)$ yields an explicit Ramsey graph. Moreover,
the bound on the independence number is replaced by the possibly
stricter requirement concerning the rank.
How difficult is it to accomplish this?
The directed quantity $\eta_1$ is easy to bound.
Shigeta and Amano \cite{shigeta2015ordered} obtained an asymptotically tight estimate, $\eta_1(n) = \Theta(n^{1/2 + o_n(1)})$.
As a brief warmup, we give the following weaker claim.

\begin{claim}
    \label{claim:eta_1_bounds}
    For every positive integer $n$, there holds 
    $n^{1/2} \le \eta_1(n) \le \mathcal{O} \left( n^{\log_4 3} \right)$.
\end{claim}
\begin{proof}
    For the lower bound, if $A$ is the adjacency matrix of a simple directed graph with no $2$-cycle,
    then $(A+I) \odot (A+I)^T = I$, and by sub-multiplicativity indeed $\rank(A+I) \ge n^{1/2}$.
    For the upper bound, note the following properties
    of $A_{D_4}$ the adjacency matrix of the directed $4$-cycle,
    \begin{itemize}
    \item
    $\rank(A_{D_4} + I) = 3$.
    \item
    $A_{D_4} + I$ has no $J_2$ principal minor.
    \end{itemize}
    Consequently $\eta_1(4) \le 3$, and
    the proof follows by repeatedly applying property $4$ of \Cref{lem:props_nu_eta}.
\end{proof}

The upper bound can be likewise improved, e.g., to 
$\eta_1(n) \le \Theta(n^{\log_6 4})$, as an exhaustive search through
digraphs of order $6$ yields $\eta_1(6) = 4$.
Turning our attention to the \textit{undirected} quantity $\nu_d(n)$,
we note that $\nu_1(n)=n$ (the graph must have no edges).
Instead, we provide, as part of the warmup, a trivial construction for $\nu_d(n)$.
Better estimates provided in the sequel require new ideas.

\begin{proposition}
    \label{prop:trivial_bound_nu}
    For every two positive integers $d$ and $n$, we have $\nu_d(n) \le \lceil \frac{n}{d} \rceil$.
\end{proposition}
\begin{proof}
    Let $G$ be the disjoint union of $\lfloor \frac{n}{d} \rfloor$ copies of $K_{d}$ and a single copy of $K_{n \bmod d}$.
    By construction, $\omega(G) = d$.
    Furthermore, $(A_G + I)$ is a block matrix of $\lceil \frac{n}{d} \rceil$ matrices J, thus $\rank(A_G + I) = \lceil \frac{n}{d} \rceil$. 
\end{proof}

\section{Rank-Ramsey Graphs and the Log-Rank Conjecture}
\label{sect:rank_ramsey_and_log_rank}

In this section we briefly overview the connections between the
log-rank problem in communication complexity, and 
Rank-Ramsey graphs. This discussion is mostly expository and not very technical.

On the one hand, we observe that any Rank-Ramsey construction \textit{witnesses} a gap in the graph-theoretic formulation
of the log-rank conjecture.
Conversely we show that, \textit{under certain conditions}, exhibiting a gap in the log-rank problem implies the construction of Rank-Ramsey graphs. 

\subsection{The Log-Rank Conjecture}

What is the relation between the \textit{rank} of a binary matrix and its \textit{communication complexity}?
As observed by Mehlhorn and Schmidt \cite{mehlhorn1982vegas},
every deterministic communication protocol induces a partition 
of the communication matrix into monochromatic rectangles. This
clearly implies that $\DCC(A) \ge \log \rank(A)$ for every binary matrix $A$,
where $\DCC(A)$ is the deterministic communication complexity of $A$.
Does a converse hold? This is one of the most notorious open questions 
in computational complexity.
The \textit{log-rank conjecture} states that 
\begin{equation}
\label{eq:log_rank}    
    \DCC(A) \le \mathcal{O} \left( \log^c \rank(A) \right) 
\end{equation}
for \textit{every} binary matrix $A$, where $c > 1$ is some absolute constant.

Our current best lower bound on $c$ is $c \ge 2$, due to G\"o\"os, Pitassi and Watson \cite{goos2015deterministic}. 
The upper bounds that we have remain exponentially far.
The current record, due to Sudakov and Tomon \cite{sudakov2023matrix} (see also \cite{lovett2016communication}),
states that $\DCC(A) = \mathcal{O}(\rank^{1/2}(A))$.

\subsection{The Graph-Theoretic Perspective}
\label{subsect:graph_theoretic_log_rank}

There are several known equivalent formulations of the log-rank conjecture,
eschewing the use of terminology from communication complexity.
For example, Nisan and Wigderson \cite{nisan1995rank} proved that
the conjecture is equivalent to the following statement:
Define $\text{mono}(A):= \max |B|/|A|$,
over all monochromatic minors $B$ of $A$.\footnote{Here we denote by
$|Z|$ the ``volume'' of the matrix $Z$, i.e., the product of $Z$'s dimensions.}
The conjecture is that for every binary matrix $A$ there holds 
\[
    -\log \left( \text{mono}(A) \right) \le \mathcal{O}( \log^{c} \rank(A) ),
\]
where $c > 1$ is a universal constant.

Here, we consider the \textit{graph-theoretic formulation}, due to Lov{\'a}sz and Saks
(see \cite{lovasz1988lattices, nisan1995rank}), which states that there exists an
absoluate constant $c>0$ such that for \textit{every} graph $G$,\footnote{
The constants $c$ appearing in these three formulations may differ. The constant of
the graph-theoretic formulation is known to be equivalent to that in a formulation
comparing rank and $\chi^1$ (i.e., the $1$-\textit{cover number}),
which, by known results \cite{aho1983notions}, is at most a factor of $2$ away from 
the constant for deterministic communication complexity.
We refer the reader to \cite{lee2023around} for details of the graph-theoretic reduction.
} 
\[
    \log \chi(G) \le \mathcal{O} \left( \log^{c}  \rank(A_{G}) \right) 
\]

To motivate the next definition, let us
think of the clique number $\omega(G)$ as the largest cardinality of a set
of vertices with independence number $1$. We
consider, likewise, the largest order of an induced subgraph
of $G$, of independence number at most $d$: 

\begin{definition}
    Let $G$ be a graph, and $d$ be a positive integer. Then,
    \[
        \psi_d(G) \eqdef \max_{\substack{U \subseteq V(G) \\ \alpha(G[U]) \le d}} |U|
    \]
\end{definition}

An interesting feature of $\psi_d(G)$ is that it 
captures the chromatic number, up to log factors. This goes back to the
classical results on integrality gaps for covering problems \cite{lovasz1975ratio}.
A similar argument appears in the analysis \cite{dietzfelbinger1996comparison},
of \textit{extended fooling sets}.\footnote{
It clearly suffices to consider the log-rank conjecture for twin-free graphs,
as \textit{reductions} affect neither chromatic number, nor rank.
The main result of \cite{kotlov1996rank} then implies one can replace the
$(\ln n)$-factor in \Cref{claim:chromatic_bound}, with $(\ln \sqrt{2} + o(1)) \rank(A_G)$.}

\begin{claim}
    \label{claim:chromatic_bound}
    For every graph $G$ of order $n$, it holds that:
    \[
        \max_{1 \le d \le \alpha(G)} \frac{\psi_d(G)}{d} \le \chi(G) \le \ln n \cdot \max_{1 \le d \le \alpha(G)} \frac{\psi_d(G)}{d} + 1
    \]
\end{claim}

\begin{proof}
    Let $H$ be an induced subgraph of $G$ attaining $\psi_d(G)$ for some
    $1 \le d \le \alpha(G)$. Then,
    \[
        \chi(G)\ge \chi(H)\ge\frac{\psi_d(G)}{\alpha(H)}\ge\frac{\psi_d(G)}{d}.
    \]
    The middle inequality follows from the fact that $\alpha(F)\cdot\chi(F)\ge v(F)$ for every graph $F$. The following colouring algorithm yields the upper bound.
\begin{itemize}
\item 
Let $U_1 := V(G)$, and
repeat until $U_{i+1}=\emptyset$:
\item
Let $U_{i+1}:=U_i\setminus S_i$,
where $S_i$ is a largest independent set in the induced subgraph $G[U_i]$.
\item 
Colour the vertices of $S_i$ with colour $i$.
\end{itemize}

If this algorithm terminates after $t$ iterations, then
it yields a $t$-colouring of $G$, whence $t\ge\chi(G)$.
Note also that $|U_i|\le\psi_{|S_i|}(G)$.
Denote $B \eqdef \max_{1 \le i \le t} |U_i| / |S_i|$,
i.e., $|S_i| \ge \frac{1}{B} |U_i|$, for $i=1, \ldots, t$.
    Thus, for all $i$, we have
    \[
        |U_{i+1}| = |U_{i}| - |S_i| \le |U_{i}| \cdot \left( 1 - \frac{1}{B} \right).
    \]
    Iterating, and using this for $i=t-1$, we get
    \[
        1 \le |U_t|\le n \cdot \left( 1 - \frac{1}{B} \right)^{t - 1} < n \cdot e^{-\frac{t-1}{B}}
    \]
    Taking logs, we have,
    \[
        \chi(G) \le t< B\cdot \ln n + 1 \le \max_{1 \le i \le t} \left( \frac{\psi_{|S_i|}(G)}{|S_i|} \right)\cdot \ln n + 1 \qedhere
    \]
\end{proof}

The following two corollaries connect between the growth rates of the functions $\nu_d(n)$, and the log-rank conjecture.
The first corollary shows that Rank-Ramsey graphs witness a 
gap in the graph-theoretic formulation
of the log-rank conjecture. Conversely, the second corollary shows how under certain circumstances, log-rank separations can imply constructions
of Rank-Ramsey graphs.

\begin{corollary}
    \label{cor:kramsey_implies_log_rank_separation}
    For any two positive integers $d$ and $n$, there exists a graph $G$ such that:
    \[
        \chi(G) \ge \frac{n}{d}, \text{ and } \rank(A_{G}) \le \nu_d(n) + 1
    \]
\end{corollary}
\begin{proof}
    Let $\overline{G}$ be an order-$n$ graph attaining $\nu_d(n)$. Then, $\chi(G) \alpha(G) \ge n$ and $\alpha(G)=\omega(\overline{G}) \le d$.
    And,
    \[
        \rank(A_{G}) = \rank(J - A_{\overline{G}} - I) \le \rank(A_{\overline{G}} + I) + 1 = \nu_d(n) + 1 \qedhere
    \]
\end{proof}

\begin{corollary}
    \label{cor:kramsey_versus_chromatic_number}
    For every graph $G$, there exists some $1 \le d \le \alpha(G)$ such that,
    \[
        \chi(G) \le \ln n \cdot \frac{\psi_d(G)}{d} + 1, \text{ and } \rank(A_G) \ge \nu_d(\psi_d(G)) - 1
    \]
\end{corollary}
\begin{proof}
    Let $1 \le d \le \alpha(G)$ be a number for which the upper bound in \Cref{claim:chromatic_bound} holds. This 
    gives us the left inequality. For the right inequality,
    let $H$ be an induced subgraph of $G$ attaining $\psi_d(G)$. Then,
    \[
        \rank(A_G) \ge \rank(A_H) = \rank(J-A_{\overline{H}}-I) \ge \rank(A_{\overline{H}}+I) - 1 \ge \nu_d(\psi_d(G)) - 1 \qedhere 
    \]
\end{proof}

\section{Rank-Ramsey Graphs from Minors of Kronecker Powers}
\label{sect:kron_powers}

In this section we obtain a polynomial separation between $\nu_d(n)$ and $n$, for all $d \ge 41$.

\begin{theorem}
    \label{thm:nu_bound_kron_power}
    The following bounds hold:
    \begin{enumerate}
        \item $\nu_{41}(n) = \mathcal{O} \left( n^{1 - \frac{1}{10000}} \right)$.
        \item For any sufficiently large $d$, $\nu_d(n) = \mathcal{O} \left( n^{\log_{296} (232)} \right)$, where $\log_{296} (232) \approx0.957 $.
    \end{enumerate}
\end{theorem}

As we know, such gaps are implied by 
the existence of low-rank Boolean matrices with the desired properties.
However, whereas most (non-explicit) constructions of Ramsey graphs
found in the literature
use the probabilistic method, this avenue is closed for us. This is due to
the fact that we have no natural distribution over low-rank Boolean matrices. 
Consequently, we opt for \textit{explicit} constructions. 

The proof of \Cref{thm:nu_bound_kron_power} is divided into two parts:
first, we find a \textit{constant} size graph $G$, beating the trivial 
bound from \Cref{prop:trivial_bound_nu} for some small $d$.
Then, we show how to locate a good principal minor within large Kronecker powers of $(A_G + I)$, such that the clique number remains low.

\subsection{Part 1: Finding a Base Graph}
\label{subsect:base_graph}

Our first goal is to find \textit{some} graph $G$, such that 
$\rank(A_G + I)$ is strictly smaller than $\frac{v(G)}{\omega(G)}$.
It is even better if $G$ has a \textit{small} clique number (say, $\omega(G)=2$),
as this improves the bounds in the blowup procedure (of part 2).
This is equivalent to seeking a graph where the 
eigenvalue $-1$ has a large multiplicity,
since $\rank(A_G + I) = v(G) - \mu_{A_G}(-1)$.

An exhaustive search reveals that there is no such graph with fewer than $10$ 
vertices. Therefore, we seek larger, highly structured graphs. 
\textit{Strongly-regular} graphs are natural candidates for such a search. 
These are regular,
highly symmetric graphs with only \textit{three} distinct eigenvalues
(the Perron eigenvalue, equal to the degree, and two other distinct eigenvalues):

\begin{definition}
    \label{defn:srg}
    Let $v > k > \max\{\lambda, \mu \} \ge 0$ be integers.
    We say that $G$ is an $\srg(v,k,\lambda,\mu)$ (i.e., strongly regular graph with parameters $v$, $k$, $\lambda$ and $\mu$) if
    \begin{enumerate}
        \item $G$ has order $v$.
        \item It is $k$-regular.
        \item Every two adjacent vertices in $G$ have $\lambda$ common neighbours.
        \item Every two non-adjacent vertices in $G$ have $\mu$ common neighbours.\footnote{
        Here we use the standard notation $\mu$ for the number of common neighbours of non-adjacenct neighbours in an SRG.
        This is not to be confused with $\mu_A(\lambda)$, the notation used for the multiplicity of $\lambda$ in $A$.
        }
    \end{enumerate}
\end{definition}

We cannot use strongly regular graphs as base graphs for $d=2$,
because apart from $K_2$ and the complete bipartite graphs $K_{n,n}$ (whose spectrum is \textit{symmetric}, and
thus cannot beat the trivial bound of \Cref{prop:trivial_bound_nu}), 
only seven triangle-free strongly regular graphs are known:
$C_5$, Petersen, Clebsch, Hoffman-Singleton, Gewirtz, Mesner, and the Higman-Sims graph.
Whether this list is exhaustive or not is a famous open question (Cf.\ \cite{biggs1993algebraic, brouwer2012strongly}).
But even if such a graph exists, another problem arises.

\begin{proposition}
    \label{prop:minus_one_tf_srg}
    $G = K_2$ is the only triangle-free strongly-regular graph for which $-1 \in \spec (G)$.
\end{proposition}
\begin{proof}
    Say $G$ is in $\srg(v, k, \lambda, \mu)$, let $A$ be its adjacency matrix, and let $r > s$ be its two distinct non-Perron eigenvalues. Observe that $\lambda = 0$ since $G$ is triangle-free. The following relation among the parameters of strongly-regular graph are well known (e.g., \cite{brouwer2012strongly}):
    \[
        (1)\ \ k-\mu = -rs,\ \ (2)\ \ \lambda = \mu + r + s, \text{ and }\ \ (3)\ \ (v-k-1)\mu = k(k-\lambda-1)
    \]
    If $r=-1$ then $(1) \implies k-\mu = s < r < 0$, which implies that $k < \mu$, a contradiction. Alternatively, if $s=-1$, then $(2) \implies \mu = 1-r < 2$ and either $\mu = 0$ (and $(3) \implies k=1$, therefore $G=K_2$), or $\mu=1$ (and $r=0$). In the latter case, it follows (see \cite{deutsch2001strongly})
    that the graph must be a Moore graph, since substituting back into $(3)$ yields:
    \[
        v = k^2 + 1 = 1 + k \sum_{i=0}^{\diam(G)-1}  (k-1)^i,
    \]
    thus $G$ is a \textit{Moore graph} (which is clearly connected of diameter two, as it is strongly-regular). It remains to rule out the Moore graphs. Firstly, $G$ can be neither the complete graph, nor an odd cycle over $>5$ vertices, as they are not strongly-regular. Moreover, $G$ cannot be $C_5$, the Petersen graph, the Hoffman-Singleton graph, or the famed (hypothesised) $57$-regular Moore graph, as their spectra are all fully determined (c.f. \cite{biggs1993algebraic, godsil2001algebraic}), and do not match $r=0, s=-1$.
\end{proof}

\Cref{prop:minus_one_tf_srg} \textit{appears} to rule out the use of strongly regular 
graphs as base graphs, if we insist on $d=2$.
However, we have a way out:
A triangle-free strongly regular graph \textit{can} have the eigenvalue $1$.
This suggests that we find a way to negate 
the spectrum of such graphs, while retaining triangle-freeness.
As the following two simple lemmas show, this is indeed possible. 

\begin{lemma}
    \label{lem:clique_kron}
    For any two graphs, $G$ and $H$, $\omega(G \otimes H) \le \min\{\omega(G), \omega(H)\}$.
\end{lemma}
\begin{proof}
Let $S \subseteq V(G) \times V(H)$ be a clique in $G \otimes H$.
By definition of the Kronecker product, the projections of $S$ onto $G$ is a clique, 
and likewise for $H$.
As the graphs are simple, every vertex appears in these projections exactly once.
\end{proof}

\begin{lemma}
    \label{lem:neg_spec}
    Let $A$ be the adjacency matrix of a graph $G$ and let $l > 2$ be an integer. Then:
    \[
        \mu_{A \otimes K_l} (-1) = (l-1) \cdot \mu_A (1)
    \]
\end{lemma}
\begin{proof} 
    Recall that the spectrum of the Kronecker Product of two matrices is the pairwise product of their respective eigenvalues. Since
    $\spec(A_{K_l}) = \{ l^{(1)}, -1^{(l-1)}$\}, it follows that
    \[
        \mu_{A \otimes K_l} (-1) = (l-1) \cdot \mu_A (1) + \mu_A \left( \frac{1}{l-1} \right)
    \]
    but $\mu_A (\frac{1}{l-1}) = 0$, since $A$ is symmetric and binary, and thus 
    all its eigenvalues are algebraic integers. 
\end{proof}

In combination, \Cref{lem:clique_kron} and \Cref{lem:neg_spec} yield
some base-graphs and upper bounds on $\nu_2$, and $\nu_3$.

\begin{corollary}
    \label{cor:nu_2_3_small_bounds}
    For any $l > 2$, we have $\nu_2(16l) \le 6l + 10$, and $\nu_3(27l) \le 7l + 20$.
\end{corollary}
\begin{proof}
    Let $\mathcal{C}$ be the Clebsch graph, and let $\mathcal{S}$ be the Schl\"afli graph. 
    It is known that:
    \[
        \omega(\mathcal{C}) = 2,\ \ \mu_{A_\mathcal{C}}(1) = 10,\ \  v(\mathcal{C}) = 16
        \text{\quad and \quad} \omega(\overline{\mathcal{S}}) = 3,\ \  \mu_{A_{\overline{\mathcal{S}}}}(1) = 20,\ \ v(\overline{\mathcal{S}}) = 27
    \]
    Consider the graph $G = \mathcal{C} \otimes K_l$.
    By construction, $v(G) = v(\mathcal{C}) \cdot v(K_l) = 16l$, and 
    by \Cref{lem:neg_spec}, it follows that $\mu_{A_G}(-1) = \mu_{A_\mathcal{C}}(1) \cdot (l-1) = 10(l-1)$, and therefore $\rank(A_G + I) = v(G) - \mu_{A_G}(-1) = 6l + 10$.
    Moreover, $\mathcal{C}$ is triangle-free, and thus by \Cref{lem:clique_kron}, so is $G$. Similarly $\overline{\mathcal{S}}$ yields the claim about $\nu_3$.
\end{proof}

\begin{figure}[H]
    \vspace{-0.1in}
    \centering
    \label{fig:clebsch_k3}
    \scalebox{0.81}{
    \begin{tikzpicture}[scale=0.85, every node/.style={inner sep=0,outer sep=0}]
        \tikzset{vertex/.style = {shape=circle,draw,minimum size=0.5em}}
        \tikzset{edge/.style = {-}}
        \pgfdeclarepattern{
            name=hatch,
            parameters={\hatchsize,\hatchangle,\hatchlinewidth},
            bottom left={\pgfpoint{-.1pt}{-.1pt}},
            top right={\pgfpoint{\hatchsize+.1pt}{\hatchsize+.1pt}},
            tile size={\pgfpoint{\hatchsize}{\hatchsize}},
            tile transformation={\pgftransformrotate{\hatchangle}},
            code={
                \pgfsetlinewidth{\hatchlinewidth}
                \pgfpathmoveto{\pgfpoint{-.1pt}{-.1pt}}
                \pgfpathlineto{\pgfpoint{\hatchsize+.1pt}{\hatchsize+.1pt}}
                \pgfpathmoveto{\pgfpoint{-.1pt}{\hatchsize+.1pt}}
                \pgfpathlineto{\pgfpoint{\hatchsize+.1pt}{-.1pt}}
                \pgfusepath{stroke}
                }
        }
        \tikzset{
            hatch size/.store in=\hatchsize,
            hatch angle/.store in=\hatchangle,
            hatch line width/.store in=\hatchlinewidth,
            hatch size=5pt,
            hatch angle=0pt,
            hatch line width=.2pt,
        }
        
        \node[vertex] (v1) at  (16.0000000000000,5.09335938000000) {};
        \node[vertex, minimum size=3em, dashed, color=gray] (v2) at  (16.0000000000000,9.75273438000000) {};
        \node[vertex, color=black, pattern=hatch, hatch size=2pt, hatch angle=0, pattern color=red] (v2_2) at  (16.0000000000000,10.05273438000000) {};
        \node[vertex, color=black, pattern=hatch, hatch size=2pt, hatch angle=0, pattern color=orange] (v2_1) at  (15.7000000000000,9.48273438000000) {};
        \node[vertex, color=black, pattern=hatch, hatch size=2pt, hatch angle=0, pattern color=black!60!green] (v2_3) at  (16.3000000000000,9.48273438000000) {};
        \node[vertex] (v3) at  (17.6421875000000,7.35546875000000) {};
        \node[vertex] (v4) at  (18.7394531200000,1.32421875000000) {};
        \node[vertex] (v5) at  (13.3421875000000,4.22929688000000) {};
        \node[vertex, minimum size=3em, dashed, color=gray] (v6) at  (11.5691406200000,6.53320312000000) {};
        \node[vertex, color=black, pattern=hatch, hatch size=2pt, hatch angle=0, pattern color=orange] (v6_1) at  (11.5691406200000,6.83320312000000) {};
        \node[vertex, color=black, pattern=hatch, hatch size=2pt, hatch angle=0, pattern color=red] (v6_2) at  (11.8691406200000,6.28320312000000) {};
        \node[vertex, color=black, pattern=hatch, hatch size=2pt, hatch angle=0, pattern color=black!60!green] (v6_3) at  (11.2691406200000,6.28320312000000) {};
        \node[vertex] (v7) at  (16.0000000000000,2.29804688000000) {};
        \node[vertex] (v8) at  (20.4308593800000,6.53320312000000) {};
        \node[vertex] (v9) at  (13.2605468800000,1.32421875000000) {};
        \node[vertex] (v10) at  (18.6578125000000,4.22929688000000) {};
        \node[vertex] (v11) at  (14.3578125000000,2.83164062000000) {};
        \node[vertex] (v12) at  (13.3421875000000,5.95781250000000) {};
        \node[vertex] (v13) at  (14.3578125000000,7.35546875000000) {};
        \node[vertex] (v14) at  (16.0000000000000,7.88906250000000) {};
        \node[vertex] (v15) at  (18.6578125000000,5.95781250000000) {};
        \node[vertex] (v16) at  (17.6421875000000,2.83164062000000) {};
        
        \draw[edge] (v1) to (v3);
        \draw[edge] (v1) to (v5);
        \draw[edge] (v1) to (v7);
        \draw[edge] (v1) to (v10);
        \draw[edge] (v1) to (v13);
        \draw[edge] (v2) to (v5);
        \draw[edge, color=blue, opacity=0.5] (v2_1) to (v6_2);
        \draw[edge, color=blue, opacity=0.5] (v2_1) to (v6_3);
        \draw[edge, color=blue, opacity=0.5] (v2_2) to (v6_1);
        \draw[edge, color=blue, opacity=0.5] (v2_2) to (v6_3);
        \draw[edge, color=blue, opacity=0.5] (v2_3) to (v6_1);
        \draw[edge, color=blue, opacity=0.5] (v2_3) to (v6_2);
        \draw[edge] (v2) to (v8);
        \draw[edge] (v2) to (v10);
        \draw[edge] (v2) to (v14);
        \draw[edge] (v3) to (v4);
        \draw[edge] (v3) to (v6);
        \draw[edge] (v3) to (v14);
        \draw[edge] (v3) to (v15);
        \draw[edge] (v4) to (v5);
        \draw[edge] (v4) to (v8);
        \draw[edge] (v4) to (v9);
        \draw[edge] (v4) to (v16);
        \draw[edge] (v5) to (v11);
        \draw[edge] (v5) to (v12);
        \draw[edge] (v6) to (v7);
        \draw[edge] (v6) to (v9);
        \draw[edge] (v6) to (v12);
        \draw[edge] (v7) to (v8);
        \draw[edge] (v7) to (v11);
        \draw[edge] (v7) to (v16);
        \draw[edge] (v8) to (v13);
        \draw[edge] (v8) to (v15);
        \draw[edge] (v9) to (v10);
        \draw[edge] (v9) to (v11);
        \draw[edge] (v9) to (v13);
        \draw[edge] (v10) to (v15);
        \draw[edge] (v10) to (v16);
        \draw[edge] (v11) to (v14);
        \draw[edge] (v11) to (v15);
        \draw[edge] (v12) to (v13);
        \draw[edge] (v12) to (v15);
        \draw[edge] (v12) to (v16);
        \draw[edge] (v13) to (v14);
        \draw[edge] (v14) to (v16);
        
    \end{tikzpicture}}
    \caption{Illustration of the Kronecker product $\mathcal{C} \otimes K_3$, where $\mathcal{C}$ is the Clebsch graph.
    Similarly to \textit{blowup}, vertices are replaced by anticliques, and edges by bipartite graphs.
    The key difference is that here edges are replaced by complete bipartite graphs \textit{minus} the identity matching (so the graph is twin-free). }
\end{figure}
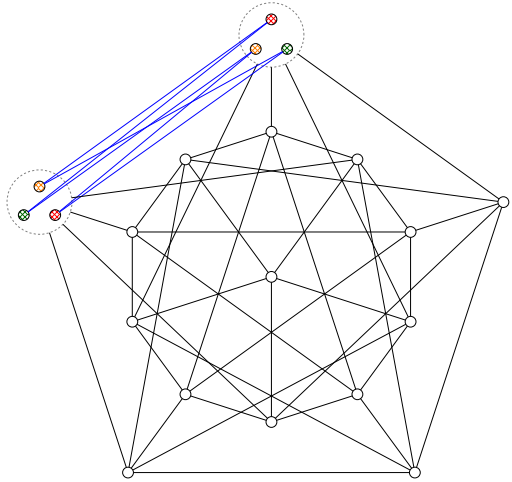

\begin{remark}
    In a very recent paper, Bamberg, Bishnoi, Ihringer and Ravi \cite{bamberg2024ramsey} constructed a family of Cayley graphs 
    of $\mathbb{F}_{2^n}^2$ that improves upon the constant of our triangle-free construction, lowering the bound from 
    $\nu_2(n) \le (3/8 + o(1))n$ (see \Cref{cor:nu_2_3_small_bounds}), to $\nu_2(n) \le (1/4 + o(1)) n$. 
\end{remark}

\subsection{Part 2: Principal Minors of Kronecker Powers}

Next we use the base-graphs from Part 1 to construct infinite 
families of graphs with
bounded clique numbers, which exhibit a polynomial separation between the order 
and rank of the complement. Let $G$ be
a base-graph, and let $H$ be the graph with adjacency matrix $(A_G + I)^{\otimes n} - I$. Clearly $\omega(H) = \omega(G)^n$
(take the direct product of $n$ copies of a largest clique of $G$, plus self loops).
However, as shown below, if $\omega(G)$ is sufficiently \textit{small},
then there is, nevertheless, a \textit{large principal minor} of $A_H$ with no \textit{large} cliques.

\begin{theorem}
    \label{thm:subtensor_clique_free}
    Let $G=([k], E)$ be a graph, $d > 1$ a positive integer, and $r := \rank\left(A_G+I\right)$.
    Then, for every large enough $n$ there exists a graph $H$ such that:
    \begin{enumerate}
        \item $\omega(H) \le d$.
        \item $\rank(A_H + I) \le r^n$
        \item $v(H) = \frac{k^n}{d} \left( \sum_{t=1}^{\omega(G)} \#K_t(G) \cdot t^{d+1} \right)^{-\frac{n}{d+1}} - 1$, where $\#K_t(G)$ is the number of copies of $K_t$ in $G$.
    \end{enumerate} 
\end{theorem}
\begin{proof}
    Let $M = (A_G + I)^{\otimes n}$, and let $X_1, \dots, X_N  \overset{\mathrm{iid}}{\sim} [k]^n$ be indices in $M$,
    sampled independently and uniformly at random (the value of $N$ will be fixed later). We find an upper bound on
    the probability that the principal (generalised) minor $W = M[X_1, \dots, X_N]$ of $M$,
    contains a principal $(d + 1) \times (d + 1)$ all-ones minor. By union bound:
    
    \[
        \Pr \left[ \exists S \in \binom{[N]}{d+1}: W[\{X_i\}_{i \in S}] = J_{d+1} \right]
        \le N^{d+1} \cdot \Pr_{X_1, \dots, X_{d+1} \sim [k]^n} \left[ W[X_1, \dots, X_{d+1}] = J_{d+1} \right]. \]
The choice of $X_1, \dots, X_{d+1}$ determines
a $n \times (d+1)$ matrix with entries 
$x_{i,j}\in [k]$. \\
    By definition of the
    Kronecker product the above event occurs exactly when
        \[ \forall i,j \in [d+1]: \forall l \in [n]: (A_G + I)_{x_{l, i}, x_{l,j}} = 1 \]
    and by independence, the probability of this event equals   
        \[ \Pr_{z_1, \dots, z_{d+1} \sim [k]} \left[ \forall i,j \in [d+1]: (A_G + I)_{z_i, z_j} = 1 \right] ^n.\]

    The latter expression counts the ways to place integer weights on $G$'s 
    vertices (with self-loops) such that all vertices in the support (i.e., with non-zero weights) form a clique.
    This is easily bounded by:\footnote{An inclusion-exclusion refinement of the following argument yields only 
    a negligible improvement in the results.}
    \[
        \Pr_{z_1, \dots, z_{d+1} \sim [k]} \left[ \forall i,j \in [d+1]: (A_G + I)_{z_i, z_j} = 1 \right] \le \frac{1}{k^{d+1}} \cdot \sum_{t=1}^{\omega(G)} \#K_t(G) \cdot t^{d+1}.
    \]
    
    If this expression is bounded away from $1$, then
    a good minor $W$ exists. When this is so, a bound on $N$ follows:
    \[
        N^{d+1} \cdot \left(\frac{1}{k^{d+1}} \cdot \sum_{t=1}^{\omega(G)} \#K_t(G) \cdot t^{d+1} \right)^n < 1 \implies N \le \left( \frac{k^{d+1}}{\sum_{t=1}^{\omega(G)} \#K_t(G) \cdot t^{d+1} } \right)^{\frac{n}{d+1}}
    \]
    
    We can, therefore, \textit{fix} some values for $X_1, \dots, X_N$, such that $W=M[X_1, \dots, X_N]$ contains no all-ones principal minor of size $d+1$. By \Cref{prop:rank_of_gen_minor}, $\rank(W) \le \rank(M) = r^n$,
    since $W$ is a generalised minor of $M$. 
    The proof is concluded by letting $A_H = W-I$.
\end{proof}

We are now ready to prove \Cref{thm:nu_bound_kron_power}.

\begin{proof}
    Fix two constants $d > 1$ and $l>2$. As in the proof of \Cref{cor:nu_2_3_small_bounds},
    let $G = \mathcal{C} \otimes K_l$, where
    $\mathcal{C}$ is the Clebsch graph. As shown there,
    \[
        v(G) = v(\mathcal{C}) \cdot v(K_l) = 16l,\ e(G) = e(\mathcal{C}) \cdot e(K_l) = 40 \binom{l}{2},\text{ and } \rank(A_G + I) = 6l + 10
    \]
    By applying \Cref{thm:subtensor_clique_free} to $G$, we obtain, for every large $N$
    a $K_{d+1}$-free graph $H$ with:
    \[
        \rank(A_H+I) \le (6l+10)^N, \text{ and } v(H) = \Theta \left( \frac{(16 l)^{d+1}}{ 16l + 40 \binom{l}{2} \left(2^{d+1} - 2\right) } \right)^{\frac{N}{d+1}}
    \]
    Denoting $n=v(H)$ and re-arranging, we may write:
    \[
        \rank(A_H + I) \le \mathcal{O} \left( n^\delta \right), \text{ where } \delta = { \frac{(d+1) \ln (6l + 10)}{(d+1) \ln(16l) - \ln \left(16l + 40 \binom{l}{2} (2^{d+1} - 2) \right) } }
    \]
    For the first bound, take $d=41$ and $l=26$. 
    
    \noindent    
    For the second bound, note that for any fixed $l$, the expression for $\delta$ decreases with $d$.
    \[
        \lim_{d\to\infty}\delta = \lim_{d\to\infty} { \frac{(d+1) \ln (6l + 10)}{(d+1) \ln(16l) - \ln \left(16l + 40 \binom{l}{2} (2^{d+1} - 2) \right) } } = \frac{\ln(6l+10)}{\ln(16l)-\ln 2}. 
    \]
    The minimum of this expression over the integers occurs at $l=37$. This yields,
    \[
        \delta = \frac{\ln(232)}{\ln(296)} + o_d(1)
    \]
    as claimed.
\end{proof}

\section{Rank-Ramsey Graphs from Lifts of Erd\H{o}s-R\'enyi Graphs}
\label{sect:ae_lift}

So far, we have constructed Rank-Ramsey graphs with constant $d$. In this section we proceed
to the range $d = \Theta(\log n)$, and prove the following theorem:

\begin{theorem}
    \label{thm:d_log_n_bound_nu}
    For any two constants $c, \varepsilon > 0$ with $c > 2 \left(\frac{2}{3\varepsilon}+1\right)^2$,
    there holds
    $
        \nu_{ c \log n}(n) \le \widetilde{\mathcal{O}} \left(n^{\frac{2}{3} + \varepsilon} \right)
    $.
\end{theorem}

These are graphs with a \textit{logarithmic} clique number,
and approaching $\rank(A_G + I)=\mathcal{O}\left(n^{2/3}\right)$.
For example, one can instantiate \Cref{thm:d_log_n_bound_nu} to obtain:
\[
    \nu_{10^6 \log n}(n) \le \mathcal{O} \left( n^{\frac{2}{3} + \frac{1}{1000}} \right)
\]

The key ingredient in our proof of \Cref{thm:d_log_n_bound_nu} is our matrix-theoretic 
view of \textit{lifting} (\Cref{subsect:matrix_view_lifts}).

\subsection{The Functions \texorpdfstring{$\NAE$}{NAE} and \texorpdfstring{$\AllEq$}{AE}}

The graphs constructed in \Cref{thm:d_log_n_bound_nu} are the lifts of matrices with the following Boolean functions.

\begin{definition}
    \label{defn:nae}
    The Boolean functions $\NAE: \{0,1\}^3 \to \{0,1\}$ and $\AllEq: \{0,1\}^3 \to \{0,1\}$ are defined by:
    \[
        \NAE(x_1, x_2, x_3) \eqdef \begin{cases}
            0 & x_1 = x_2 = x_3 \\
            1 & \text{otherwise}
        \end{cases}, \text{ and } \AllEq \eqdef 1 - \NAE
    \]
\end{definition}

The seminal paper of Nisan and Wigderson \cite{nisan1995rank}
exhibited a gap for the log-rank conjecture
using recursive compositions of $\NAE$ and lifts thereof (with the $\AND$-gadget).
For more on this, see \Cref{sect:nw_analysis}.

Since the $\NAE$ function depends only on the Hamming weight of the input,
its unique expression as a multilinear polynomial is a linear combination of elementary symmetric functions,
\[
    \NAE(x_1,x_2,x_3) = e_2(x_1, x_2, x_3) - e_1(x_1, x_2, x_3) + 1 = x_1 x_2 + x_2 x_3 + x_1 x_3 - x_1 - x_2 - x_3 + 1. 
\]

\subsection{Our Construction}
\label{subsect:ae_lift_construction}

It is more convenient to formulate the proof in terms of $\AllEq$, 
rather than $\NAE$. 
Our aim is to bound from above two quantities of the graphs that we construct: 
(i) Rank, and (ii) Maximum Clique Size. 

\paragraph{The Rank.} The polynomial representation of $\AllEq$ (in particular, its low degree), in conjunction with \Cref{lem:rank_of_lift},
imply that the lift of \textit{any} three matrices $A_1, A_2, A_3 \in M_k$, 
\[M \eqdef \AllEq(A_1, A_2, A_3) \in M_{k^3}(\RR)\]
has \textit{low rank}, namely $\rank(M) \le 3k^2 + 3k = \mathcal{O}(k^2)$.

\paragraph{The Clique Number.}
In our construction $A_1$, $A_2$ and $A_3$ are adjacency matrices of simple graphs,
$G_1$, $G_2$ and $G_3$ respectively. This guarantees
that $M$ is a symmetric binary matrix with ones on the diagonal.
As an illustration (and an aside), this class of graphs includes the \textit{trivial construction} of \Cref{prop:trivial_bound_nu}:
if $G_1 = K_k$ and $G_2=G_3=\overline{K}_k$, the resulting lifted graph is a disjoint union of cliques.\footnote{
Given graphs $G_1, G_2, G_3$ of \textit{small order}, lifts $\AllEq(G_1, G_2, G_3)$
can be used as base-graphs for \Cref{thm:subtensor_clique_free}.
For instance, $M = \AllEq(K_4, K_4, M_2) \in M_{64}(\RR)$, where $M_2$ is a perfect matching, has
$\omega(G) = 2$ and $\rank(A_G) = 31 < \sfrac{64}{\omega(G)} = 32$. 
However, this does not beat the base-graph of \Cref{subsect:base_graph}.}

Consider a set of vertices
$S = \left\{ (x_1, y_1, z_1), \dots, (x_s, y_s, z_s) \right\}$ in $V(G_1) \times V(G_2) \times V(G_3)$.
By definition of $\AllEq$, this set
forms a clique in the lifted graph if and only if the graphs
(with multiplicities, i.e., generalised minors) induced over the constituent graphs are \textit{identical}.
Therefore, informally, cliques of the lifted graph emerge from \textit{correlations} between $G_1$, $G_2$ and $G_3$.
To avoid correlations, we will sample
our base graphs i.i.d.\ from $G(k, 1/2)$, as
in the classical lower bound on diagonal Ramsey numbers.

\paragraph{Large Projections.} We view the vertices of the lifted graph as lattice
points in $[k] \times [k] \times [k]$. To control the clique size of the lifted graph,
we seek to avoid subsets 
$S \subseteq [k] \times [k] \times [k]$ with
small axis projections $|\pi_x(S)|$, $|\pi_y(S)|$ and $|\pi_z(S)|$. Since our
constituent graphs are random, 
sets with large projections will guarantee ``sufficient randomness'', and 
will therefore be less likely to induce cliques in the lifted graph. 
This intuition is made concrete in the proof of \Cref{thm:d_log_n_bound_nu}.
Beforehand, we use a probabilistic argument to show that there exist \textit{large} subsets of the cube,
in which \textit{every} subset above a certain cardinality has \textit{some} large
axis projection.\footnote{
The Loomis-Whitney inequality provides
some lower bounds on projection sizes. Namely, for every 
$S \subset [k] \times [k] \times [k]$, we have
$|\pi_{xy}(S)| |\pi_{yz}(S)| |\pi_{xz}(S)| \ge |S|^2$ and $|\pi_{x}(S)| |\pi_{y}(S)| 
|\pi_{z}(S)| \ge |S|$.}

\begin{lemma}
    \label{lem:large_proj_cube_subset}
    Let $a > 0$ and $1 > \eta >0$ be constants.
    For any sufficiently large $k$, there exists a subset $T \subset [k]^3$,
    of cardinality $|T| = \widetilde{\Theta} \left( k^{3(1-\eta)} \right)$, such that for every subset $S \subset T$ of cardinality $|S| = a \log k$,
    \[
    \max \left\{ |\pi_x(S)|, |\pi_y(S)|, |\pi_z(S)| \right\} \ge \eta|S|= \eta a \log k
    \]
\end{lemma}
\begin{proof}
    Let $T$ be a sample of $N$ points, chosen i.i.d and uniformly from $[k]^3$, with $N$ to be determined later. What is
    the probability that $T$ fails to satisfy the statement in the lemma?
    Say that a subset $S\subset T$ of cardinality $|S|=a \log k$ is $x$-{\it bad} if $|\pi_x(S)| \le \eta a \log k$.
    Similarly define $y$ and $z$-bad subsets.
    Call $S$ {\it bad} if it is $x, y$ and $z$-bad. A sample $T$ with no bad subsets is called {\em appropriate}. 
    If the expected number of bad subsets $S\subset T$ is negligible, then by Markov's inequality, whp $T$ is appropriate.

    The event that $S$ is bad at a certain axis (e.g. $x$-bad) is equivalently formulated as follows: 
    We throw $a \log k$ balls independently into $k$ bins, 
    and all our balls fall into at most $\eta a \log k$ bins.
    The probability of this bad event
    is \[\binom{k}{\eta a \log k}\left(\frac{\eta a\log k}{k}\right)^{a\log k}\le k^{(\eta - 1) a \log k + \mathcal{O}(\log \log k)} .\]
    So, the probability that $S$ is bad is at most
    \[ k^{3(\eta - 1) a \log k + \mathcal{O}(\log \log k)} .\]
    Consequently, the expected number of bad subsets that are contained in
    our random set $T$ is at most
    \[ \binom{N}{a \log k} k^{3(\eta - 1) a \log k + \mathcal{O}(\log \log k)} .\]
    This expression is $o(1)$ provided that
    \[N=k^{3\left(1-\eta\right)-o(1)}\]
    in which case we can conclude that there \textit{exists} 
    an appropriate sample $T$ of size $N$.
    
    Since $T$ is chosen {\em with repetitions}, we have to account for the possibility
    that the same point is selected more than once. But no point can
    be chosen more than $a(1-\eta) \log k$ times. For
    otherwise any set of $a \log k$ triples containing these duplicates would
    be bad and render $T$ inappropriate. 
    It follows that there exists an appropriate set $T' \subset [k]^3$ of cardinality at least 
    $\frac{N}{a(1-\eta) \log k}$.
\end{proof}

Appropriate sets allow us to identify a large principal minor of the lifted graph,
in which every set of cardinality $a \log k$ is ``sufficiently random''.
With it, we are ready to present our construction.
\begin{proof}[Proof of Theorem \ref{thm:d_log_n_bound_nu}]

    As mentioned, our construction proceeds by sampling three graphs 
    $G_1$, $G_2$, and $G_3$, independently and uniformly at random from 
    $G \left(k, 1/2 \right)$. Let
    $A_1$, $A_2$ and $A_3$ be their respective adjacency matrices, and let $M \eqdef \AllEq( A_1, A_2, A_3 )$ be their lift. 
    
    Pick $a > 0$ and $1>\eta >0$, construct a
    subset $T \subset [k]^3$ as in \Cref{lem:large_proj_cube_subset},
    and consider $W \eqdef M[T]$.
    Since $G_1$, $G_2$ and $G_3$ have no self-loops, 
    the main diagonal of $M$ (and hence $W$ as well) is all-ones.
    Also, $$\rank(W) \le \rank(M) \le 3k^2 + 3k$$ in view of \Cref{lem:rank_of_lift}
    and the fact that $\AllEq$ has total degree $2$.\\

    So, we now have the {\em rank} under control and move to deal with the {\em clique number}.
    It remains to show that, with positive probability (over the choice of $G_1$, $G_2$ and $G_3$),
    the matrix $W$  has no all-ones principal minors of order $a \log k$.
    Let $X$ be the random variable that counts the number of such principal minors. 
    If its expectation is $\mathbb{E}(X)=o(1)$, then whp the resulting graph has no $(a \log k)$-clique. 
    
    Consider a subset $S \subset T$ of cardinality $|S|=a\log k$.
    What is the probability that $W[S]$ is an all-ones matrix?
    According to \Cref{lem:large_proj_cube_subset},
    $S$ has a large projection on some axis,
    say $|\pi_x(S)| \ge \eta a \log k$ and for convenience we assume that the first $\eta a \log k$ $x$'s are all distinct.
    By definition of the $\AllEq$ function, the adjacency relations
    among the $\eta a \log k$ \textit{distinct} vertices of $\pi_x(S)$ and the corresponding vertices of $\pi_y(S),\pi_z(S)$, must coincide with each other.
    That is,
    \[
        A_1\left[ x_1, \dots, x_{\eta a \log k} \right] = A_2\left[ y_1, \dots, y_{\eta a \log k} \right] = A_3\left[ z_1, \dots, z_{\eta a \log k} \right].
    \]
    Since $A_1[x_1, \dots, x_{\eta a \log k}]$ is a \textit{proper} minor 
    of $A_1$, with no repeated indices,
    it is the adjacency matrix of a $G(\eta a \log k, 1/2)$ graph. 
    Therefore, for any fixed choice of $A_2\left[ y_1, \dots, y_{\eta a \log k} \right]$,
    it holds that the event $A_1\left[ x_1, \dots, x_{\eta a \log k} \right] = A_2\left[ y_1, \dots, y_{\eta a \log k} \right]$
    occurs with probability exactly
    $$
    2^{-\binom{\eta a \log k}{2}}=k^{-\left(\frac{\eta^2 a^2}{2} - o(1)\right)\log k}
    $$
    Consequently 
    \[
        \mathbb{E}(X) \le \binom{|T|}{a\log k} k^{-\left(\frac{\eta^2 a^2}{2} - o(1)\right)\log k} = k^{ \left( 3a(1-\eta) - \frac{\eta^2 a^2}{2} + o(1) \right) \log k }.
    \]
    It follows that $\mathbb{E}(X) = o(1)$ when $a>\frac{6(1-\eta)}{\eta^2}$.
    The proof now follows by fixing 
    \[
        n \eqdef |T| = \widetilde{\Theta} \left(k^{3(1 - \eta)} \right),\ \ c \eqdef \frac{a}{3(1 - \eta)},\ \ \varepsilon \eqdef \frac{2\eta}{3(1 - \eta)} \qedhere
    \]    
\end{proof}

\subsection{Rank-Ramsey Graphs with Logarithmic Clique Number}

In \Cref{subsect:ae_lift_construction} we construct Rank-Ramsey graphs with \textit{logarithmic} clique number.
and \textit{polynomial} rank, approaching $n^{2/3}$.
Conversely, in the study of ``classical'' Ramsey graphs,
Erd\H{o}s \cite{erdos1947some} showed that asymptotically almost all $n$-vertex graphs
have logarithmic clique \textit{and} independence numbers, and this is indeed best-possible up to constants
(c.f. \cite{erdos1935combinatorial, thomason1988upper, campos2023exponential}).
This contrast between Ramsey and Rank-Ramsey graphs raises the question:
what is the growth rate of $\nu_{a \log n}(n)$?

Our construction witnesses that for every constant $a>0$, 
there holds $\nu_{a \log n}(n) = \widetilde{\mathcal{O}}(n^{2/3} + K(a))$,
where $K(a)$ is a constant depending on $a$.
Conversely, we remark that if the log-rank conjecture holds,
then by \Cref{cor:kramsey_implies_log_rank_separation}, there exists a universal constant $c > 1$,
such that for every constant $a>0$,
\[
    \nu_{a \log n}(n) = \Omega \left( 2^{\sqrt[c]{\log (\frac{n}{a \log n})}} \right)
\]

\section{In Search of Better Bounds on \texorpdfstring{$\nu_2(n)$}{nu2(n)}}
\label{sect:lb_sect}

The Ramsey numbers $R(3,n)$ have fascinated combinatorialists for decades.
Culminating a long line of excellent research,
Kim \cite{kim1995ramsey} proved that $R(3,n)=\Theta(n^2 / \log n)$.
Even the implicit constant is known, up to a factor of $(4 + o(1))$ \cite{fiz2020triangle}.
What about $R^k(3,n)$, or equivalently, $\nu_2(n)$? Trivially,
\[
    \Omega \left( \sqrt{n \log n} \right) \le \nu_2(n) \le \left\lfloor\frac{n}{2}\right\rfloor 
\]
where the lower-bound follows from the Ramsey numbers $R(3,n)$.

The base-graphs used in \Cref{sect:kron_powers}, derived from the Clebsch graph, yields
an improvement of the \textit{upper bound} to $\nu_2(n) \le (3/8 + o(1))n$.
If we wish to improve the \textit{lower bound} to
$\nu_2(n) \ge \Omega(n^{1/2 + \varepsilon})$ for some positive constant $\varepsilon$,
then of course, bounds on the \textit{independence number} will not do,
as there \textit{exist} triangle-free graphs with $\alpha(G) = \Omega(\sqrt{n \log n})$ \cite{ajtai1980note}.
Instead, we turn to consider other graph parameters. More concretely 
we investigate parameters that are
related to \textit{orthonormal representations} of graphs, on which
a considerable body of research exists. This choice
is motivated by the fact that the Gram
matrix that corresponds to an orthonormal representation of $G$ is reminiscent of $A_G+I$.

\begin{definition}
    \label{defn:ortho_rep}
    An orthonormal representation $\left\{ w_v \in \RR^N \right\}_{v \in V(G)} $ of a graph $G$
    is an assignment of unit vectors to the vertices of $G$, such that
    \[
        \forall v,u \in V(G):\ \langle w_v, w_u \rangle = \begin{cases}
            1 & v=u \\
            0 & v \not \sim u \\
            \star & \text{otherwise}
        \end{cases}
    \]
\end{definition}
The Gram matrix of the vectors $w_v$
in any orthnormal representation of an $n$-vertex graph $G$, is an $(n \times n)$ 
positive semidefinite matrix $W$
that \textit{agrees} with $A_G+I$ on the main diagonal, and on $G$'s non-edges.
It is very suggestive that such matrices can tell us a lot about $A_G + I$.
Two of the most studied measures related to $M$
are the Lov{\'a}sz number $\vartheta(G)$,
and the minimum semidefinite rank, $\msr(G)$, whose definitions we recall below.
Like $\rank(A_G + I)$, these two quantities are related to constructions of Ramsey graphs. 
By the well-known ``sandwich theorem'' (see \cite{knuth1993sandwich}),
for every graph $G$ there holds
\[
    \alpha(G) \le \vartheta(G) \le \msr(G) \le \chi(\overline{G})
\]

How does $\rank(A_G+I)$ fit into this web of relations? We show the following.

\begin{theorem}
    For infinitely many $n>1$, there exist triangle-free graphs 
    $G_1$ and $G_2$ of order $n$, with
    \begin{enumerate}
        \item $\vartheta(G_1) = \Theta(n^{2/3})$ and $\rank(A_{G_1} + I) = n$.
        \item $\msr(G_2) \ge n/2$ and $\rank(A_{G_2} + I) = (3/8 + o(1))n$.
    \end{enumerate}
\end{theorem}

That is, we show \textit{separations} between $\rank(A_G + I)$, and $\vartheta(G)$ and $\msr(G)$.\footnote{
Triangle-free graphs with $\rank(A_G+I)$ much larger than $\msr(G)$ are also 
not hard to come by. For example, the complete bipartite graph $G=K_{n,n}$ 
has $\msr(G) \le \chi(\overline{G}) = \chi(K_n \sqcup K_n) = n$
and $\rank(A_G + I) = 2n$.} Understanding the relations between $\vartheta(G)$ and $\rank(A_G + I)$
is particularly interesting in view of striking similarities between
the two parameters. See \Cref{subsect:lovasz_vs_rank} for more.

It is conceivable to us that $\vartheta(G) = \mathcal{O}\left(\rank(A_G + I)\right)$.
For instance, $\vartheta(G) \le \rank(A_G+I)$ for every
graph of order less than $10$, and we do not know any graph that fails this inequality.
We remark that if this indeed holds true, then the results of 
\cite{kashin1981systems, alon1998approximating} regarding the Lov{\'a}sz number
of graphs with bounded independence number, imply a \textit{polynomial improvement}
on the lower bounds on $\nu_d$, namely we have
$\nu_d(n) = \Omega \left( n^{2/(d+1)} \right)$,
versus the weaker bound $\Omega(n^{1/d})$ (see \Cref{cor:nu_vs_ramsey}) 
that follow purely from Ramsey numbers.
We refer the reader to \Cref{sect:hoffman_cvetkovic_lovasz_ramsey} for a discussion
on the connection between this open problem, and other well-known problems regarding the relationship 
between the well-known Hoffman's bound and Cvetkovi\'c's bound.

\subsection{Lov{\'a}sz Number versus Rank}
\label{subsect:lovasz_vs_rank}

Let us recall the definition of the Lov{\'a}sz number $\vartheta(G)$ of a simple graph $G$:

\begin{definition}[\cite{lovasz1979shannon}]
    \label{defn:lovasz_number}
    The Lov{\'a}sz number of a graph $G=(V,E)$ is
    \[
        \vartheta(G) \eqdef \min_{c, W} \max_{v \in V} \frac{1}{\langle c, w_v \rangle^2}
    \]
    where $W: V\to \RR^N$ is an orthonormal representation of $G$ and $c \in \RR^N$ is a unit vector.
\end{definition}

The graph parameters $\vartheta(G)$ and $\rank(A_G + I)$ share many properties. E.g.,
both are multiplicative in the strong product (see \Cref{defn:graph_kron}), and both are upper bounds
on the Shannon capacity of a graph.

\begin{proposition}
    Let $G$ and $H$ be two graphs. Then,
    \[
        \vartheta(G \boxtimes H) = \vartheta(G) \vartheta(H), \text{ and } \rank(A_{G \boxtimes H} + I) = \rank(A_G + I) \rank(A_H + I)
    \]
\end{proposition}
\begin{proof}
    Lov{\'a}sz proved (\cite{lovasz1979shannon}, Lemma 2 and Theorem 7) that $\vartheta$ is \textit{multiplicative} in the strong product.
    As for the rank, by \Cref{defn:graph_kron} and the multiplicativity of rank under the Kronecker product, 
    \begin{align*}
        \rank(A_{G \boxtimes H} + I) &= \rank( A_G \otimes A_H + A_G \otimes I + I \otimes A_H + I \otimes I ) \\
        &= \rank( (A_G + I) \otimes (A_H + I) ) = \rank(A_G + I) \rank(A_H + I) \qedhere
    \end{align*}
\end{proof}
\begin{corollary}[See also \cite{lovasz1979shannon, haemers1979some}]
    For every graph $G$ it holds that
    \[
        \vartheta(G) \ge \Theta(G) \text{ and } \rank(A_G + I) \ge \Theta(G)
    \]
    where $\Theta(G) \eqdef \sup_k \sqrt[k]{ \alpha(G^{\boxtimes k}) }$ is the \textit{Shannon capacity} of $G$.
\end{corollary}

\subsubsection{Best-Possible Separation of \texorpdfstring{$\rank(A_G + I)$}{rank(A + I)} from \texorpdfstring{$\vartheta(G)$}{theta(G)}}

The current best known \textit{explicit construction} of triangle-free Ramsey graphs
is due to Alon \cite{alon1994explicit}.
An order-$n$ graph $G$ in this family satisfies $\vartheta(G) = \Theta(n^{2/3})$.
As we will presently show, there also holds that $\rank(A_{G} + I) = n$. 
In other words, these are \textit{best-possible} $\vartheta$-Ramsey graphs, yet
they are \textit{worst-possible} Rank-Ramsey graphs. \\

\noindent Let us briefly recount Alon's construction: 
Let $k > 1$ be such that $3 \nmid k$ and let\footnote{
This refers to the binary representation of elements in $\GF(2^k)$, which are naturally identified with binary strings of length $k$, as usual.}
\begin{alignat*}{2}
    &W_0 \eqdef \{ x \in \GF(2^k) : \text{ leftmost bit of $x^7$ is $0$ } \},\quad &&W_1 \eqdef \{ x \in \GF(2^k) : \text{ leftmost bit of $x^7$ is $1$ } \} \\
    \text{ and }\quad\quad\quad&\\
    &U_0 \eqdef \{ (w, w^3, w^5) : w \in W_0 \} \subset \mathbb{Z}_2^{3k}, &&U_1 \eqdef \{ (w, w^3, w^5) : w \in W_1 \} \subset \mathbb{Z}_2^{3k}
\end{alignat*}
Then, $G_k \eqdef \Cay \left( \ZZ_2^{3k}, \{ u_0 + u_1 : u_0 \in U_0, u_1 \in U_1 \} \right)$ is a graph on $n_k \eqdef |\ZZ_2^{3k}|=2^{3k}$ vertices.

\begin{theorem}[\cite{alon1994explicit}]
    For every $k > 1$ with $3 \nmid k$, $G_k$ is triangle-free, and $\vartheta(G_k) = \Theta( n_k^{2/3} )$.
\end{theorem}

This bound on $\vartheta$ is tight, as shown by Kashin and Konyagin.

\begin{theorem}[\cite{kashin1981systems}]
    For every $n$ vertex triangle-free graph $G$, there holds $\vartheta(G) \ge  (2n)^{2/3}$.
\end{theorem}
\begin{proof}
    Kashin and Konyagin \cite{kashin1981systems} (see also \cite{alon1994explicit}) proved
    $\vartheta(\overline{G}) \le 2^{2/3} n^{1/3}$ for every triangle-free graph.
    The claim follows, since  
    $\vartheta(G) \vartheta(\overline{G}) \ge n$ for every graph of order $n$
    (see \cite{lovasz1979shannon}, Corollary 2).
\end{proof}

In the following claim, we show that $\rank(A_{G_k} + I) = n_k$, for every admissible $k$.

\begin{claim}
    For every $k > 1$ with $3 \nmid k$, there holds $\rank(A_{G_k} + I) = n_k$.
\end{claim}
\begin{proof}
    Since $G_k$ is a Cayley graph of an Abelian group, $\ZZ_2^{3k}$, its eigenvalues correspond to sums of group
    characters, evaluated over the generating set. The characters of $\ZZ_2^t$ are the Fourier-Walsh
    functions, i.e., the \textit{parity functions} $\chi_D: Z_2^t \to \{ \pm 1 \}$, for every set $D \subseteq [t]$. 
    So, every $D \subseteq [n]$, yields an eigenvalue $\lambda_D$ of $A_{G_k}$, where
    \[
        \lambda_D = \sum_{\substack{u_0 \in U_0 \\ u_1 \in U_1}} \chi_D(u_0 + u_1) = \left( \sum_{u \in U_0} \chi_D(u) \right) \left( \sum_{u \in U_1} \chi_D(u) \right)
    \]
    which is an even integer, because
    $|U_1| = |W_1| = 2^{k-1}$ (see \cite{alon1994explicit}) are \textit{even}.
    Consequently, all eigenvalues of $A_{G_k}$ are even, and 
    $\rank(A_{G_k} + I) = n_k - \mu_{A_{G_k}}(-1) = n_k$.
\end{proof}

\subsection{Minimum Semidefinite Rank}

Another obvious point of comparison is the \textit{rank}.
The minimum semidefinite rank, $\msr(G)$, is the smallest rank of an $n \times n$
orthonormal representation matrix of an order-$n$ graph $G$.
It is easy to see that triangle-free graphs have a large $\msr(G)$.\footnote{
For instance, a trivial Cauchy-Schwarz argument implies that 
$\msr(G) \ge n/3$ for every triangle-free $G$.
}
Deaett \cite{deaett2011minimum} showed that $\msr(G) \ge n/2$, for any connected triangle-free graph.
In contrast, in \Cref{cor:nu_2_3_small_bounds} we prove that $\nu_2(16l) \le 6l + 10$, which is attained
by a family of \textit{connected} graphs.
Therefore,

\begin{corollary}
    For any sufficiently large $n$, there exists an order $n$ triangle-free graph $G$, with 
    \[
        \msr(G) \ge n/2 \text{ and } \rank(A_G + I) = (3/8 + o(1))n
    \]
\end{corollary}

We do not know whether a converse of the form $\rank(A_G+I) = \Omega(\msr(G))$ holds.

\section{The Nisan-Wigderson Construction}
\label{sect:nw_analysis}

In a seminal paper, Nisan and Wigderson \cite{nisan1995rank} 
constructed an infinite family of symmetric binary matrices $A_k$
of dimension $2^{3^k}$. These matrices
exhibit a separation between log-rank and communication complexity, and show
that $c \ge \log_2 3$ in the conjectured inequality (\ref{eq:log_rank}).

As discussed in \Cref{sect:rank_ramsey_and_log_rank},
such separations can be related to Rank-Ramsey graphs. Here, we analyse the Nisan-Wigderson 
matrices from this perspective,
and show that these matrices $A_k$ have \textit{very large monochromatic principal minors}
and yield, therefore, poor Rank-Ramsey graphs.

\begin{claim}
    For every $k \ge 1$ there exist subsets $S,T \subset [2^{3^k}]$ of cardinality $2^{3^k (1-o_k(1))}$ each, such that  
    \[
        A_k[S] = \mathbf{0}, \text{ and } A_k[T] = J
    \]
\end{claim}

We stress that this does \textit{not} a contradict
the discussion of \Cref{sect:rank_ramsey_and_log_rank}:
Exhibiting a log-rank separation \textit{does not} preclude a matrix from \textit{having} a large monochromatic rectangle.

\subsection{The Construction}

The construction of \cite{nisan1995rank} involves a \textit{lift} using the $\AND$ gadget,
with a recursive composition of the function $\NAE$ (as in \Cref{defn:nae}) with itself.
To describe their result, we require some notation.

\begin{definition}
    For every $k > 1$, the Boolean function $\NAE^k: \{0,1\}^{3^k} \to \{0,1\}$ is defined:
    \[
        \NAE^k(x) \eqdef \NAE \left( \NAE^{k-1}(x_1, \dots, x_{3^{k-1}}), \NAE^{k-1}(x_{3^{k-1} + 1}, \dots, x_{2 \cdot 3^{k-1}}), \NAE^{k-1}(x_{2 \cdot 3^{k-1} + 1}, \dots, x_{3^k}) \right) 
    \]
    and $\NAE^1 \eqdef \NAE$.
\end{definition}
We denote by $A_k$ the symmetric binary matrix $\NAE^k \circ \AND^{3^k}$. Moreover:
\[
    \forall b \in \{0,1\}:\ G^{(b)}_k \eqdef G \left( \left\{ x \in \{0,1\}^{3^k} : \NAE^k(x) = b \right\}, \left\{ \{x,y\} : \NAE^k(x \land y) = 1 \right\} \right)
\]
in other words, $G^{(1)}_k$ (resp.\ $G^{(0)}_k$) is the graph whose adjacency matrix is the principal minors of $A_k$,
induced by those indices for which the main diagonal is $1$ (resp.\ $0$).
With this notation, the result of Nisan and Wigderson says:\footnote{One can even
replace $\DCC(A_k)$ by $\mathrm{N}^0(A_k)$ (the $0$-nondeterministic communication complexity of $A_k$) in the statement of the theorem,
as the bound on $\DCC(A_k)$ follows from a reduction to the promise communication problem ``unique disjointness'',
for which Kaibel and Weltge \cite{kaibel2015short} recently proved the bound $\chi^0( \mathrm{UDISJ}_n ) \ge (3/2)^n$,
where $\chi^0$ is the $0$-cover number.}

\begin{theorem}[\cite{nisan1995rank}]
    \label{thm:nw_theorem}
    Let $A_k \eqdef \NAE^k \circ \AND^{3^k} \in M_{3^k}(\RR)$. Then,
    \[
        \log \rank(A_k) \le \mathcal{O} (2^k), \text{ and } \DCC(A_k) = 3^k (1 - o_k(1)) 
    \]
\end{theorem}

It easy to bound $\text{rank}(A_k)$:
It is well known (e.g., \cite{knop2021log})
that $\rank(f \circ \AND^n) = \spar(f)$ for 
any function $f: \{0,1\}^n \to \{0,1\}$.
So here the rank is the number of monomials appearing in the expansion of $\NAE^{k}$,
which can be bounded by a simple inductive argument.

\subsection{Finding Large Monochromatic Principal Minors}

In search of large monochromatic principal minors of $A_k$, let us first estimate the orders of
the two subgraphs, $G^{(0)}_k$ and $G^{(1)}_k$.

\begin{proposition}
    For every $k \ge 1$, we have that:
    \[
        v\left(G^{(0)}_k\right) = 2^{3^k} \cdot \left(\frac{1}{3} \pm o_k(1) \right), \text{ and } v\left(G^{(1)}_k\right) = 2^{3^k} \cdot \left(\frac{2}{3} \pm o_k(1) \right)
    \]
\end{proposition}
\begin{proof}
    Let $p_k$ be the probability that a uniformly random $x \sim \{0,1\}^{3^k}$ is in $V(G^{(1)}_k)$. Then,
    \begin{align*}
        p_k &= \Pr_{x \sim \{0,1\}^{3^k}} \left[ \left(\NAE^k \circ \AND^{3^k}\right) (x) = 1 \right] \\
        &= \Pr_{x_1, x_2, x_3 \sim \{0,1\}^{3^{k-1}}} \left[ \NAE \left(\NAE^{k-1}(x_1), \NAE^{k-1}(x_2), \NAE^{k-1}(x_3) \right) = 1 \right] \\
        &= 3 p_{k-1}^2(1-p_{k-1}) + 3 p_{k-1} (1-p_{k-1})^2 = 3 p_{k-1} (1-p_{k-1}).
    \end{align*}

    By direct observation we have $p_1 = 3/4$. Furthermore, 
    $x>3x(1-x)>\tfrac 23$ for $\tfrac 34 > x > \tfrac 23$. Thus,
    the sequence $p_k$ is decreasing to its limit,
    the unique positive root of $x=3x(1-x)$, namely $x=2/3$.
\end{proof}

We now construct large monochromatic principal minors in $G^{(0)}_k$ and $G^{(1)}_k$.

\begin{proposition}
    The matrices $A_k$ have large monochromatic principal minors. That is,
    \[
        \min \left\{ \omega\left(G^{(1)}_k\right), \alpha\left(G^{(0)}_k \right) \right\} \ge 2^{3^k \left( 1 - o_k(1) \right)} 
    \]
\end{proposition}
\begin{proof}
    For every positive $k$, let us denote
    \[
        \alpha_k \eqdef \alpha \left(G^{(0)}_{k} \right), \text{ and }\omega_k \eqdef \omega \left( G^{(1)}_k \right)
    \]
    We claim that $\alpha_k \ge \omega_{k-1}^3$ and that $\omega_k \ge 2^{3^{k-1}} \alpha_{k-1} \omega_{k-1}$, for any $k \ge 2$.
    Let $A,B \subset \{0,1\}^{3^{k-1}}$ be a largest clique of $G^{(1)}_{k-1}$ and largest anticlique of $G^{(0)}_{k-1}$, respectively.
    Observe that for every $x_1, x_2 \in A$, it holds that $\NAE^{k-1}(x_1 \land x_2) = 1$ (whether or not $x_1 = x_2$).
    Likewise it holds that $\NAE^{k-1}(y_1 \land y_2) = 0$ for every $y_1, y_2 \in B$. Therefore, the sets
    \[
        B' \eqdef A \times A \times A \subset \{0,1\}^{3^k}, \text{ and } A' \eqdef \{0,1\}^{3^{k-1}} \times A \times B \subset \{0,1\}^{3^k}
    \]
    are an anticlique of $G^{(0)}_k$, and a clique of $G^{(1)}_k$, respectively.

    Combining the two bounds, it holds that
    $\omega_k \ge 2^{3^{k-1}} \alpha_{k-1} \omega_{k-1} \ge 2^{3^{k-1}} \omega_{k-1} \omega_{k-2}^3$.
    Taking logs and denoting $a_k \eqdef \log (\omega_k)$, we thus arrive at the linear recurrence $a_k = 3^{k-1} + a_{k-1} + 3 a_{k-2}$.
    Denoting $a_k=3^k(1-\varepsilon_k)$, the above translates into
    \[3^k(1-\varepsilon_k)=3^{k-1}+3^{k-1}(1-\varepsilon_{k-1})+3^{k-1}(1-\varepsilon_{k-2})\]
    i.e.,
    \begin{equation}\label{eqn:vareps}
    \varepsilon_k=\frac{\varepsilon_{k-1}+\varepsilon_{k-2}}{3}   
    \end{equation}
    and we conclude that $a_k=3^k (1 - o_k(1))$, where the little-oh term is exponentially
    small in $k$. To see this consider $\lambda^2-\frac{\lambda+1}{3}$, 
    the characteristic polynomial of Equation \ref{eqn:vareps}, the roots of which are
    $\frac{1 \pm \sqrt{13}}{6}$.
 \end{proof}

\section{Open Problems}
\label{sect:open_problems}

Many intriguing questions regarding the Rank-Ramsey problem remain unanswered.
For starters, the growth rate of the function $\nu_d(n)$ is mostly unknown.
For bounded $d$, we have shown (see \Cref{sect:kron_powers}) a polynomial separation
between $\nu_d(n)$ and $n$, starting at $d=41$.
Is this true of all $d$?
Do there exist triangle-free Rank-Ramsey graphs?
Concretely,

\begin{openq}
    Is there a constant $c > 0$ such that $\nu_2(n) = \mathcal{O}(n^{1-c})$?
\end{openq}

The paper \cite{linial2007complexity} advocates the perspective that rank
is a complexity measure of sign matrices and mentions some additional
measures of similar nature like $\gamma_2$, margin complexity and more.
Rather than demand that the complement rank be small,
one can similarly investigate graphs with low clique number for which 
$\gamma_2(\overline{G})$ is also small etc.\footnote{
If one is tempted to replace \textit{both} the clique number with rank, 
\textit{and} the independence number with complement rank, this renders the
problem uninteresting, as $\max \{\rank(A_G), \rank(A_G+I)\} \ge n/2$ (either $0$ or $-1$ has multiplicity $\le n/2$).
}

We saw several connections between Rank-Ramsey numbers and various
graph parameters. Our list is far from exhausting all possible interesting ties.
Could there be any relation with the Colin de Verdi{\`e}re parameter \cite{de1993new})?
With other orthonormal representations of graphs
(e.g., see \cite{laurent2012gram})?

In \Cref{sect:lb_sect}, we briefly consider minimum semidefinite rank,
and show $n$-vertex triangle-free graphs $G$ and $H$, with 
both $\left(\msr(G) - \rank(A_G + I)\right) = \Omega(n)$
and with $\left(\rank(A_H + I) - \msr(H)\right) = \Omega(n)$.
The relation between the 
two is therefore likely nuanced.
We find it interesting to 
understand what is the least distance, under the rank metric,
between $A_G + I$ and a PSD representation matrix $M$ for $G$.
Another quantity which we consider in \Cref{sect:lb_sect}, originating in
orthonormal representations, is the Lov{\'a}sz number.
We give some evidence that the Lov{\'a}sz number
bounds the complement rank from below, perhaps up to a multiplicative constant.
Thus, we ask:

\begin{openq}
    What is the relation between $\vartheta(G)$ and $\rank(A_G + I)$ for a simple graph $G$?
\end{openq}

The regime of unbounded $d$, which we explore in \Cref{sect:ae_lift}
is also of great interest. 
Recall that we construct a graph of logarithmic clique number, and polynomial rank. 
This is  in stark contrast to the classical Ramsey problem, 
where in almost all graphs both the
clique number \textit{and} independence number are logarithmic.
Naturally, we think that maintaining low-rank should be much harder than controlling the independence number.
It is therefore quite natural to ask,

\begin{openq}
    What is the growth rate of $\nu_d(n)$, as a function of $d$ and $n$ $(\text{as } d=d(n) \to \infty)$?
\end{openq}

In \Cref{subsect:rank_ramsey} we begin an exploration of the Rank-Ramsey numbers. 
Unlike the usual Ramsey numbers, which are symmetric (i.e., $R(s,t) = R(t,s)$ for every $s$ and $t$), 
this does not hold for our numbers.
In fact, we determine the numbers $R^k(s,t)$ for every $2 \le t \le 5$, and 
prove that $R^k(3,n) > R^k(n,3)$ for sufficiently large $n$. 
It would be interesting to better understand the interplay between the clique number and complement rank.

Another perspective of Rank-Ramsey numbers, stems from \textit{twin-free graphs}.
Recall that two vertices in a graph
are called {\em twins} if they are non-adjacent and have the same set of neighbours.
Pruning twins from a graph clearly affects neither the rank and nor chromatic number.
Therefore, the sets $\mathcal{G}_r$ of all twin-free connected graphs of rank $r$,
play a crucial role in understanding the log-rank conjecture.
Indeed, the log-rank problem (in particular, its graph-theoretic formulation, see \Cref{sect:rank_ramsey_and_log_rank}) can be \textit{equivalently} re-formulated as follows:
\paragraph{Equivalent Formulation of Log-Rank Problem.} \ \vspace{0.05in} \\
Is there a constant $c > 0$ such that 
$
\displaystyle \log \chi(G) \le \mathcal{O} \left( \log^c r \right)
\text{~for every~}G \in \mathcal{G}_r\text{~and every~} r > 1 ?
$ \\

It would therefore be interesting to investigate the sets $\mathcal{G}_r$, for a larger range of values $r$.
\bibliography{rank_ramsey}

\newcommand{\etalchar}[1]{$^{#1}$}
\begin{thebibliography}{GLM{\etalchar{+}}16}

\bibitem[AK98]{alon1998approximating}
Noga Alon and Nabil Kahale.
\newblock Approximating the independence number via the $\vartheta$-function.
\newblock {\em Mathematical Programming}, 80(3):253--264, 1998.

\bibitem[AKS80]{ajtai1980note}
Mikl{\'o}s Ajtai, J{\'a}nos Koml{\'o}s, and Endre Szemer{\'e}di.
\newblock A note on {R}amsey numbers.
\newblock {\em Journal of Combinatorial Theory, Series A}, 29(3):354--360, 1980.

\bibitem[Alo94]{alon1994explicit}
Noga Alon.
\newblock Explicit {R}amsey graphs and orthonormal labelings.
\newblock {\em The Electronic Journal of Combinatorics}, pages R12--R12, 1994.

\bibitem[AUY83]{aho1983notions}
Alfred~V Aho, Jeffrey~D Ullman, and Mihalis Yannakakis.
\newblock On notions of information transfer in {VLSI} circuits.
\newblock In {\em Proceedings of the fifteenth annual ACM symposium on Theory of computing}, pages 133--139, 1983.

\bibitem[BBI24]{bamberg2024ramsey}
John Bamberg, Anurag Bishnoi, and Ferdinand Ihringer.
\newblock {R}amsey numbers and extremal structures in polar spaces.
\newblock {\em arXiv preprint arXiv:2406.03043}, 2024.

\bibitem[BF22]{babai2020linear}
L{\'a}szl{\'o} Babai and P{\'e}ter Frankl.
\newblock Linear algebra methods in combinatorics.
\newblock 2022.

\bibitem[BH12]{brouwer2012strongly}
Andries~E Brouwer and Willem~H Haemers.
\newblock Strongly regular graphs.
\newblock {\em Spectra of graphs}, pages 115--149, 2012.

\bibitem[Big93]{biggs1993algebraic}
Norman Biggs.
\newblock {\em Algebraic Graph Theory}.
\newblock Cambridge University Press, Cambridge, second edition, 1993.

\bibitem[CCD93]{chung1993note}
Fan~RK Chung, Richard Cleve, and Paul Dagum.
\newblock A note on constructive lower bounds for the ramsey numbers $r(3, t)$.
\newblock {\em Journal of Combinatorial Theory, Series B}, 57(1):150--155, 1993.

\bibitem[CF92]{calderbank1992improved}
A.~Robert Calderbank and Peter Frankl.
\newblock Improved upper bounds concerning the {E}rd{\H{o}}s-{K}o-{R}ado theorem.
\newblock {\em Combinatorics, Probability and Computing}, 1(2):115--122, 1992.

\bibitem[CGMS23]{campos2023exponential}
Marcelo Campos, Simon Griffiths, Robert Morris, and Julian Sahasrabudhe.
\newblock An exponential improvement for diagonal {R}amsey.
\newblock {\em arXiv preprint arXiv:2303.09521}, 2023.

\bibitem[CHY11]{chang2011characterization}
Gerard~J Chang, Liang-Hao Huang, and Hong-Gwa Yeh.
\newblock A characterization of graphs with rank 4.
\newblock {\em Linear algebra and its applications}, 434(8):1793--1798, 2011.

\bibitem[CHY12]{chang2012characterization}
Gerard~J Chang, Liang-Hao Huang, and Hong-Gwa Yeh.
\newblock A characterization of graphs with rank 5.
\newblock {\em Linear algebra and its applications}, 436(11):4241--4250, 2012.

\bibitem[CKLM19]{chattopadhyay2019simulation}
Arkadev Chattopadhyay, Michal Kouck{\`y}, Bruno Loff, and Sagnik Mukhopadhyay.
\newblock Simulation theorems via pseudo-random properties.
\newblock {\em computational complexity}, 28:617--659, 2019.

\bibitem[CPR00]{codenotti2000some}
Bruno Codenotti, Pavel Pudl{\'a}k, and Giovanni Resta.
\newblock Some structural properties of low-rank matrices related to computational complexity.
\newblock {\em Theoretical Computer Science}, 235(1):89--107, 2000.

\bibitem[Cve71]{cvetkovic1971graphs}
Drago{\v{s}}~M Cvetkovi{\'c}.
\newblock Graphs and their spectra.
\newblock {\em Publikacije Elektrotehni{\v{c}}kog fakulteta. Serija Matematika i fizika}, (354/356):1--50, 1971.

\bibitem[Dea11]{deaett2011minimum}
Louis Deaett.
\newblock The minimum semidefinite rank of a triangle-free graph.
\newblock {\em Linear algebra and its applications}, 434(8):1945--1955, 2011.

\bibitem[DF01]{deutsch2001strongly}
J~Deutsch and PH~Fisher.
\newblock On strongly regular graphs with $\mu = 1$.
\newblock {\em European Journal of Combinatorics}, 22(3):303--306, 2001.

\bibitem[DHS96]{dietzfelbinger1996comparison}
Martin Dietzfelbinger, Juraj Hromkovi{\v{c}}, and Georg Schnitger.
\newblock A comparison of two lower-bound methods for communication complexity.
\newblock {\em Theoretical Computer Science}, 168(1):39--51, 1996.

\bibitem[DV93]{de1993new}
Yves~Colin De~Verdi{\`e}re.
\newblock On a new graph invariant and a criterion for planarity.
\newblock {\em Contemporary Mathematics}, 147:137--137, 1993.

\bibitem[Erd47]{erdos1947some}
Paul Erd{\"o}s.
\newblock {Some remarks on the theory of graphs}.
\newblock {\em Bulletin of the American Mathematical Society}, 53(4):292 -- 294, 1947.

\bibitem[ES35]{erdos1935combinatorial}
Paul Erd{\"o}s and George Szekeres.
\newblock A combinatorial problem in geometry.
\newblock {\em Compositio mathematica}, 2:463--470, 1935.

\bibitem[FPGM20]{fiz2020triangle}
Gonzalo Fiz~Pontiveros, Simon Griffiths, and Robert Morris.
\newblock {\em The triangle-free process and the {R}amsey number $R(3, k)$}, volume 263.
\newblock Memoirs of the American mathematical society, 2020.

\bibitem[FW81]{frankl1981intersection}
Peter Frankl and Richard~M. Wilson.
\newblock Intersection theorems with geometric consequences.
\newblock {\em Combinatorica}, 1:357--368, 1981.

\bibitem[GG55]{greenwood1955combinatorial}
Robert~E Greenwood and Andrew~Mattei Gleason.
\newblock Combinatorial relations and chromatic graphs.
\newblock {\em Canadian Journal of Mathematics}, 7:1--7, 1955.

\bibitem[GLM{\etalchar{+}}16]{goos2016rectangles}
Mika G{\"o}{\"o}s, Shachar Lovett, Raghu Meka, Thomas Watson, and David Zuckerman.
\newblock Rectangles are nonnegative juntas.
\newblock {\em SIAM Journal on Computing}, 45(5):1835--1869, 2016.

\bibitem[God03]{godsil2003interesting}
Chris Godsil.
\newblock Interesting graphs and their colourings.
\newblock {\em Unpublished notes}, 2003.

\bibitem[GPW15]{goos2015deterministic}
Mika G{\"o}{\"o}s, Toniann Pitassi, and Thomas Watson.
\newblock Deterministic communication vs. partition number.
\newblock In {\em 2015 IEEE 56th Annual Symposium on Foundations of Computer Science}, pages 1077--1088. IEEE, 2015.

\bibitem[GPW17]{goos2017query}
Mika G{\"o}{\"o}s, Toniann Pitassi, and Thomas Watson.
\newblock Query-to-communication lifting for {BPP}.
\newblock In {\em 2017 IEEE 58th Annual Symposium on Foundations of Computer Science}, pages 132--143. IEEE, 2017.

\bibitem[GR01]{godsil2001algebraic}
Chris Godsil and Gordon~F Royle.
\newblock {\em Algebraic graph theory}, volume 207.
\newblock Springer Science \& Business Media, 2001.

\bibitem[Hae79]{haemers1979some}
Willem~H Haemers.
\newblock On some problems of {L}ov{\'a}sz concerning the {S}hannon capacity of a graph.
\newblock {\em IEEE Transactions on Information Theory}, 25:231--232, 1979.

\bibitem[Han51]{hanani1951number}
Haim Hanani.
\newblock On the number of straight lines determined by $n$ points.
\newblock {\em Riveon Le'Matematika}, 5:10--11, 1951.

\bibitem[Ihr23]{ihringer2023ratio}
Ferdinand Ihringer.
\newblock Ratio bound ({L}ov{\'a}sz number) versus inertia bound.
\newblock {\em arXiv preprint arXiv:2312.09524}, 2023.

\bibitem[Kim95]{kim1995ramsey}
Jeong~Han Kim.
\newblock The {R}amsey number $r(3, t)$ has order of magnitude $t^2 / \log t$.
\newblock {\em Random Structures \& Algorithms}, 7(3):173--207, 1995.

\bibitem[KK81]{kashin1981systems}
Boris~Sergeevich Kashin and Sergei~Vladimirovich Konyagin.
\newblock On systems of vectors in a {H}ilbert space.
\newblock {\em Trudy Mat. Inst. imeni VA Steklova}, 157:64--67, 1981.

\bibitem[KL96]{kotlov1996rank}
Andrew Kotlov and L{\'a}szl{\'o} Lov{\'a}sz.
\newblock The rank and size of graphs.
\newblock {\em Journal of Graph Theory}, 23(2):185--189, 1996.

\bibitem[KLMY21]{knop2021log}
Alexander Knop, Shachar Lovett, Sam McGuire, and Weiqiang Yuan.
\newblock Log-rank and lifting for {AND}-functions.
\newblock In {\em Proceedings of the 53rd Annual ACM SIGACT Symposium on Theory of Computing}, pages 197--208, 2021.

\bibitem[KN96]{kushilevitz1997communication}
Eyal Kushilevitz and Noam Nisan.
\newblock {\em Communication Complexity}.
\newblock Cambridge University Press, 1996.

\bibitem[Knu93]{knuth1993sandwich}
Donald~E Knuth.
\newblock The sandwich theorem.
\newblock {\em arXiv preprint math/9312214}, 1993.

\bibitem[KW15]{kaibel2015short}
Volker Kaibel and Stefan Weltge.
\newblock A short proof that the extension complexity of the correlation polytope grows exponentially.
\newblock {\em Discrete \& Computational Geometry}, 53:397--401, 2015.

\bibitem[KW24]{kwan2024inertia}
Matthew Kwan and Yuval Wigderson.
\newblock The inertia bound is far from tight.
\newblock {\em Bulletin of the London Mathematical Society}, 2024.

\bibitem[LMSS07]{linial2007complexity}
Nati Linial, Shahar Mendelson, Gideon Schechtman, and Adi Shraibman.
\newblock Complexity measures of sign matrices.
\newblock {\em Combinatorica}, 27:439--463, 2007.

\bibitem[Lov75]{lovasz1975ratio}
L{\'a}szl{\'o} Lov{\'a}sz.
\newblock On the ratio of optimal integral and fractional covers.
\newblock {\em Discrete mathematics}, 13(4):383--390, 1975.

\bibitem[Lov79]{lovasz1979shannon}
L{\'a}szl{\'o} Lov{\'a}sz.
\newblock On the {S}hannon capacity of a graph.
\newblock {\em IEEE Transactions on Information theory}, 25(1):1--7, 1979.

\bibitem[Lov16]{lovett2016communication}
Shachar Lovett.
\newblock Communication is bounded by root of rank.
\newblock {\em Journal of the ACM (JACM)}, 63(1):1--9, 2016.

\bibitem[LS88]{lovasz1988lattices}
L{\'a}szl{\'o} Lov{\'a}sz and Michael Saks.
\newblock Lattices, {M}\"obius functions and communications complexity.
\newblock In {\em [Proceedings 1988] 29th Annual Symposium on Foundations of Computer Science}, pages 81--90. IEEE Computer Society, 1988.

\bibitem[LS23]{lee2023around}
Troy Lee and Adi Shraibman.
\newblock Around the log-rank conjecture.
\newblock {\em Israel Journal of Mathematics}, 256(2):441--477, 2023.

\bibitem[LV12]{laurent2012gram}
Monique Laurent and Antonios Varvitsiotis.
\newblock The {G}ram dimension of a graph.
\newblock In {\em International Symposium on Combinatorial Optimization}, pages 356--367. Springer, 2012.

\bibitem[MS82]{mehlhorn1982vegas}
Kurt Mehlhorn and Erik~M Schmidt.
\newblock Las vegas is better than determinism in {VLSI} and distributed computing.
\newblock In {\em Proceedings of the fourteenth annual ACM symposium on Theory of computing}, pages 330--337, 1982.

\bibitem[NW95]{nisan1995rank}
Noam Nisan and Avi Wigderson.
\newblock On rank vs. communication complexity.
\newblock {\em Combinatorica}, 15(4):557--565, 1995.

\bibitem[RM97]{raz1997separation}
Ran Raz and Pierre McKenzie.
\newblock Separation of the monotone {NC} hierarchy.
\newblock In {\em Proceedings 38th Annual Symposium on Foundations of Computer Science}, pages 234--243. IEEE, 1997.

\bibitem[SA15]{shigeta2015ordered}
Manami Shigeta and Kazuyuki Amano.
\newblock Ordered biclique partitions and communication complexity problems.
\newblock {\em Discrete Applied Mathematics}, 184:248--252, 2015.

\bibitem[Sin18]{sinkovic2018graph}
John Sinkovic.
\newblock A graph for which the inertia bound is not tight.
\newblock {\em Journal of Algebraic Combinatorics}, 47:39--50, 2018.

\bibitem[ST23]{sudakov2023matrix}
Benny Sudakov and Istv{\'a}n Tomon.
\newblock Matrix discrepancy and the log-rank conjecture.
\newblock {\em arXiv preprint arXiv:2311.18524}, 2023.

\bibitem[Tho88]{thomason1988upper}
Andrew Thomason.
\newblock An upper bound for some {R}amsey numbers.
\newblock {\em Journal of graph theory}, 12(4):509--517, 1988.

\bibitem[Tik20]{tikhomirov2020singularity}
Konstantin Tikhomirov.
\newblock Singularity of random {B}ernoulli matrices.
\newblock {\em Annals of Mathematics}, 191(2):593--634, 2020.

\bibitem[Wen91]{wenger1991extremal}
Rephael Wenger.
\newblock Extremal graphs with no ${C}_4$'s, ${C}_6$'s, or ${C}_{10}$'s.
\newblock {\em Journal of Combinatorial Theory, Series B}, 52(1):113--116, 1991.

\end{thebibliography}
\bibliographystyle{alpha}
\appendix
\section{The Construction of Codenotti, Pudl{\'a}k and Resta}
\label{sect:2i_construction}

Codenotti, Pudl{\'a}k and Resta constructed \cite{codenotti2000some}
an interesting family of explicit triangle-free Ramsey graphs.
The construction is elementary.

\begin{theorem}[\cite{codenotti2000some}]
    \label{thm:high_mult_ev_2}
    For every sufficiently large $n$, there exists an $n$-vertex triangle-free graph $H_n$ such that
    $\rank(A_{H_n} - 2I) = \mathcal{O}\left( n^{3/4} \right)$.
\end{theorem}
\noindent The construction is as follows:
Let $G = (L \sqcup R, E)$ be a $k$-vertex bipartite graph with girth at least $8$.
Let $D = (E, \hat{E})$ be the directed graph whose vertices are the (undirected) edges of $G$, and
\[
    \forall (l,r), (a,b) \in E:\ (l, r) \to_{D} (a, b) \iff (l \sim_{G} b) \land (a \ne l) \land (b \ne r)
\]
It is not hard to see that:
\begin{enumerate}
    \item $\rank(A_{D} - I) \le k$.
    \item $D$ contains no transitive triangle.
    \item There are no back-and-forth edges in $D$.
\end{enumerate}
The undirected graph $H$ whose adjacency matrix is $A_H = A_D + A_D^T$ (i.e., omitting orientations) is therefore triangle-free,
and has $\rank(A_H - 2I) \le 2 \rank(A_D - I) \le 2k$.
It remains to fix the base graph $G$ to be bipartite, with large size relative to order, and of girth at least $8$.
For this, the graphs $H_3$ of Wenger \cite{wenger1991extremal} suffice: they 
have order $k$ and size $n=\mathcal{O}\left(k^{4/3}\right)$. 

As usual, $\alpha(H) \le \rank(A_{H} + c I)$ for any $c \ne 0$, so this Theorem yields Ramsey graphs.
While these graphs have \textit{very high} multiplicity of the eigenvalue $2$
(or $-2$, by the clique-tensoring trick of \Cref{lem:clique_kron}),
they are not Rank-Ramsey, as $\rank(A + I) \ge n - \rank(A - 2I) = (1-o(1))n$.
Crucially, we remark that the matrix $A_H - 2I$ is not binary.

\section{Hoffman, Cvetkovi\'c, Lov\'asz and (Rank-)Ramsey}
\label{sect:hoffman_cvetkovic_lovasz_ramsey}

Any Rank-Ramsey graph is also a Ramsey graph, since $\rank(A_G + I) \ge \alpha(G)$ holds for every graph $G$.
How does this bound compare to other known upper bounds on the independence number of a graph?
Two of the best-known bounds are Hoffman's ratio bound, and Cvetkovi\'c's inertia bound, 
both of which relate to the \textit{spectra} of 
\textit{symmetric weighted adjacency matrices} of a graph $G$.
\begin{definition}
    Let $G$ be a graph and let $w: E(G) \to \RR$ be a function. The weighted adjacency of $G$ is the matrix:
    \[
        A_{G,w} \eqdef \left(w(\{x, y\}) \cdot \mathbbm{1}\{x \sim y\}\right)_{x,y \in V(G)}
    \]
\end{definition}

\begin{theorem}(Hoffman Bound)
    \label{thm:hoffman}
    Let $G$ be an $n$-vertex graph and let $w: E(G) \to \RR$ be a weight function, such that $A_{G,w}$ has constant row sums.
    Then, 
    \[
        \alpha(G) \le h(A_{G,w}) \eqdef \left| \frac{\lambda_{min}(A_{G,w})}{\lambda_{max}(A_{G,w}) - \lambda_{min}(A_{G,w})} \right| \cdot n 
    \]
\end{theorem}
\begin{theorem}(Cvetkovi\'c Bound \cite{cvetkovic1971graphs, calderbank1992improved})
    \label{thm:cvetkovic}
    Let $G$ be a graph and let $w: E(G) \to \RR$ be a weight function. Then,
    \[
        \alpha(G) \le c(A_{G,w}) \eqdef \left| \left\{ \lambda \in \mathrm{spec}(A_{G,w}) : \lambda \ge 0 \right\} \right|
    \]
\end{theorem}

For a graph $G$, let $h(G) \eqdef \inf_{w} h(A_{G,w})$ and $c(G) \eqdef \inf_{w} c(A_{G,w})$
be the best-possible bounds attained by \Cref{thm:hoffman} and \Cref{thm:cvetkovic}, respectively.
The Hoffman and Cvetkovi\'c bounds had been at the centre of several longstanding problems in algebraic graph theory.
Sinkovic \cite{sinkovic2018graph} constructed the first example of
a graph for which $c(G)$ is not tight,
resolving decade-long an open question of Godsil \cite{godsil2003interesting}.
The comparability of Hoffman and Cvetkovi\'c had similarly been undetermined until very recently:
Kwan and Wigderson \cite{kwan2024inertia} constructed a family of graphs
for which $h(G) = o(c(G))$, 
and a week later, Igringer \cite{ihringer2023ratio} gave an example of a family for which $c(G) = o(h(G))$.
Both of these separations are \textit{polynomial}.

In this paper we consider two other bounds on the independence number $\alpha(G)$:
$\rank(A_G+I)$, by means of Rank-Ramsey graphs, and the Lov\'asz theta $\vartheta(G)$,
as discussed in \Cref{sect:lb_sect}.
These two bounds are intimately related to Hoffman and Cvetkovi\'c.
Clearly $\rank(A_G + I) = n - \mu_{A_G}(-1) \ge c(G)$.
Moreover, it is well-known that $\vartheta(G) \ge h(G)$, for every graph $G$.
Therefore, we may think of $\rank(A_G+I)$ as a less-tight analogue of Cvetkovi\'c's bound,
and of $\vartheta(G)$ is an analogue of Hoffman's bound, and as such the task of comparing
$\rank(A_G + I)$ and $\vartheta(G)$ related to the well-known problem of comparing the ratio
bound versus the inertia bound.

\end{document}